\documentclass[final]{siamltex}

\usepackage{amsmath,amssymb,hyperref,latexsym}
\usepackage{graphicx}
\usepackage{epsf}
\usepackage{subfigure}
\usepackage{amsfonts}
\usepackage{graphics}

\usepackage{enumerate}

 \newtheorem{cor}[theorem]{Corollary}
 \newtheorem{prop}[theorem]{Proposition}
 \newtheorem{assu}[theorem]{Assumption}

 \newtheorem{defn}[theorem]{Definition}
 \newtheorem{example}[theorem]{Example}
 \newtheorem{rem}[theorem]{Remark}
 \numberwithin{equation}{section}

\newcommand{\beq}{\begin{equation}}
\newcommand{\eeq}{\end{equation}}

\newcommand{\bB}{\mathbb{B}}

\newcommand{\bR}{\mathbb{R}}
\newcommand{\bN}{\mathbb{N}}

\newcommand{\R}{\bR}
\newcommand{\Rn}{\bR^n}
\newcommand{\Rm}{\bR^m}
\newcommand{\Rmi}{\bR^{m_i}}

\newcommand{\Rmn}{\bR^{m\times n}}
\newcommand{\Rnn}{\bR^{n\times n}}

\newcommand{\eR}{{\bar\bR}}

\newcommand{\cA}{\mathcal{A}}

\newcommand{\cG}{\mathcal{G}}

\newcommand{\cL}{\mathcal{L}}

\newcommand{\cI}{\mathcal{I}}

\newcommand{\epi}[1]{\mathrm{epi}(#1)}
\newcommand{\argmin}{\mathop{\mathrm{argmin}}}

\newcommand{\half}{\frac{1}{2}}
\newcommand{\thalf}{\tfrac{1}{2}}

\newcommand{\norm}[1]{\left\Vert #1\right\Vert}
\newcommand{\tnorm}[1]{\norm{#1}_2}

\newcommand{\onorm}[1]{\norm{#1}_1}
\newcommand{\dnorm}[1]{\norm{#1}_*}
\newcommand{\ip}[2]{\left\langle #1,\, #2\right\rangle}
\newcommand{\set}[2]{\left\{#1\,\left|\, #2\right.\right\}}
\newcommand{\map}[3]{#1:\, #2\rightarrow #3}

\newcommand{\dom}[1]{\mathrm{dom}\left(#1\right)}

\newcommand{\dist}[2]{\mathrm{dist}\left(#1\,\left|\, #2\right.\right)}
\newcommand{\odist}[2]{\mathrm{dist}_1\left(#1\,\left|\, #2\right.\right)}
\newcommand{\tdist}[2]{\mathrm{dist}_2\left(#1\,\left|\, #2\right.\right)}
\newcommand{\sdist}[2]{\mathrm{dist}_2^2\left(#1\,\left|\, #2\right.\right)}

\newcommand{\support}[2]{\delta^*\left(#1\,\left|\, #2\right.\right)}
\newcommand{\indicator}[2]{\delta\left(#1\,\left|\, #2\right.\right)}
\newcommand{\ncone}[2]{N\left(#1\, |\, #2\right)}

\newcommand{\sd}{\partial}

\newcommand{\lam}{\lambda}

\newcommand{\eps}{\epsilon}
\newcommand{\veps}{\varepsilon}
\newcommand{\tveps}{{\tilde\veps}}
\newcommand{\sig}{\sigma}
\newcommand{\alf}{\alpha}
\newcommand{\del}{\delta}

\newcommand{\by}{\bar y}
\newcommand{\bp}{\bar p}
\newcommand{\br}{\bar r}
\newcommand{\bz}{\bar z}
\newcommand{\bx}{\bar x}

\newcommand{\bk}{{\bar k}}

\newcommand{\bu}{\bar u}

\newcommand{\bw}{\bar w}

\newcommand{\beps}{{\bar\eps}}

\newcommand{\tx}{{\tilde x}}

\newcommand{\tr}{{\tilde r}}
\newcommand{\tu}{{\tilde u}}

\newcommand{\tw}{{\tilde w}}
\newcommand{\tW}{{\tilde W}}
\newcommand{\tA}{{\tilde A}}
\newcommand{\tS}{{\tilde S}}
\newcommand{\tJ}{{\tilde J}}
\newcommand{\teps}{{\tilde \eps}}

\newcommand{\hx}{{\hat x}}

\newcommand{\hb}{{\hat b}}

\newcommand{\hA}{{\hat A}}
\newcommand{\hG}{{\hat G}}
\newcommand{\hv}{{\hat v}}
\newcommand{\hu}{{\hat u}}

\newcommand{\hr}{{\hat r}}
\newcommand{\hS}{{\hat S}}
\newcommand{\hJ}{{\hat J}}
\newcommand{\heps}{{\hat \eps}}
\newcommand{\hmu}{{\hat\mu}}

\newcommand{\om}{\omega}

\newcommand{\Pri}{P_{C_i}}

\newcommand{\qri}{q_i}
\newcommand{\wri}{\tw_i}

\title    {Matrix-Free Solvers for Exact Penalty Subproblems}
\author   {James V. Burke\thanks{Dept. of Mathematics, University of Washington, Seattle, WA 98195, USA.}
      \and Frank E. Curtis\thanks{Dept. of Industrial and Systems Engineering, Lehigh University, Bethlehem, PA 18018, USA.}
      \and Hao Wang\thanks{Dept. of Industrial and Systems Engineering, Lehigh University, Bethlehem, PA 18018, USA.}
      \and Jiashan Wang\thanks{Dept. of Mathematics, University of Washington, Seattle, WA 98195, USA.}
          }
\date     {\today}

\begin{document}

\maketitle

\begin{abstract}
  We present two matrix-free methods for approximately solving exact penalty subproblems that arise when solving large-scale optimization problems.  The first approach is a novel iterative re-weighting algorithm (IRWA), which iteratively minimizes quadratic models of relaxed subproblems while automatically updating a relaxation vector.  The second approach is based on alternating direction augmented Lagrangian (ADAL) technology applied to our setting.  The main computational costs of each algorithm are the repeated minimizations of convex quadratic functions which can be performed matrix-free.   We prove that both algorithms are globally convergent under loose assumptions, and that each requires at most $O(1/\varepsilon^2)$ iterations to reach $\varepsilon$-optimality of the objective function.  
  Numerical experiments exhibit the ability of both algorithms to efficiently find inexact solutions.  Moreover, in certain cases, IRWA is shown to be more reliable than ADAL.
\end{abstract}



\section{Introduction}

The prototypical convex composite optimization problem is
\begin{equation}\label{basic cc problem}
  \min_{x\in X}\ f(x)+\dist{F(x)}{C},
\end{equation}
where the sets $X\subset\Rn$ and $C\subset\Rm$ are non-empty, closed, and convex, the functions $\map{f}{\Rn}{\R}$ and $\map{F}{\Rn}{\Rm}$ are smooth, and the distance function is defined as
\[
  \dist{y}{C} := \inf_{z\in C}\ \norm{y-z},
\]
with $\norm{\cdot}$ a given norm on $\Rm$ \cite{Burke85,Fletcher82,Powell83}.  The objective in problem \eqref{basic cc problem} is an exact penalty function for the optimization problem
\[
  \min_{x\in X}\ f(x)\ \mbox{subject to}\ F(x) \in C,
\]
where the penalty parameter has been absorbed into the distance function.  Problem \eqref{basic cc problem} is also useful in the study of feasibility problems where one takes $f\equiv 0$.

Problems of the form \eqref{basic cc problem} and algorithms for solving them have received a great deal of study over the last 30 years \cite{AnOs77,Flet87,RTRW}.  The typical approach for solving such problems is to apply a Gauss-Newton strategy to either define a direction-finding subproblem paired with a line search, or a trust-region subproblem to define a step to a new point \cite{Burke85,Powell83}.  This paper concerns the design, analysis, and implementation of methods for approximately solving the subproblems in either type of approach in large-scale settings.  These subproblems take the form
\begin{equation}\label{g-dist}
  \min_{x\in X}\ g^Tx + \thalf x^THx+\dist{Ax+b}{C},
\end{equation}
where $g\in\Rn$, $H\in\Rnn$ is symmetric, $A\in\Rmn$, $ b\in\Rm$, and $X\subset\Rn$ and $C\subset\Rm$ may be modified versions of the corresponding sets in \eqref{basic cc problem}.  In particular, the set $X$ may now include the addition of a trust-region constraint.  In practice, the matrix $H$ is an approximation to the Hessian of the Lagrangian for the problem \eqref{basic cc problem} 
\cite{Burke87,Fletcher82,Powell83}, 
and so may be indefinite depending on how it is formed.  
However, in this paper, we assume that it is positive semi-definite so that subproblem \eqref{g-dist} is convex.

To solve large-scale instances of \eqref{g-dist}, we develop two solution methods based on linear least-squares subproblems. These solution methods are matrix-free in the sense that the least-squares subproblems can be solved in a matrix-free manner.  The first approach is a novel iterative re-weighting strategy \cite{BeTu74,Lea90,Os85,Schlos73,WoS88}, while the second is based on ADAL technology \cite{BPCPE11,Eck89,Spin85} adapted to this setting.  We prove that both algorithms are globally convergent under loose assumptions, and that each requires at most $O(1/\varepsilon^2)$ iterations to reach $\varepsilon$-optimality of the objective of \eqref{g-dist}.  We conclude with numerical experiments that compare these two approaches.

As a first refinement, we suppose that $C$ has the product space structure
\begin{equation}\label{prod C}
  C := C_1 \times \cdots \times C_l,
\end{equation}
where, for each $i \in \cI := \{1,2,\dots,l\}$, the set $C_i\subset\Rmi$ is convex and $\sum_{i\in\cI} m_i = m$.  Conformally decomposing $A$ and $b$, we write
\[
  A=\begin{bmatrix}A_1\\ \vdots\\ A_l\end{bmatrix}
  \quad\mbox{and}\quad
  b=\begin{pmatrix}b_1\\ \vdots\\ b_l\end{pmatrix},
\]
where, for each $i\in\cI$, we have $A_i\in\R^{m_i\times n}$ and $b_i\in\Rmi$.  On the product space $\R^{m_1}\times \cdots\times \R^{m_l}$, we define a norm adapted to this structure as
\begin{equation}\label{new norm}
  \norm{(y^T_1,y^T_2,\dots, y^T_l)^T}:=\sum_{i\in\cI}\tnorm{y_i}.
\end{equation}
It is easily verified that the corresponding dual norm is
\[
\dnorm{y}=\sup_{i\in\cI}\tnorm{y_i}.
\]
With this notation, we may write
\begin{equation}\label{prod dist}
  \dist{y}{C} = \sum_{i\in\cI}\tdist{y_i}{C_i},
\end{equation}
where, for any set $S$, we define the distance function $\tdist{y}{S}:=\inf_{z\in S}\tnorm{y-z}$.  Hence, with $\varphi(x) := g^Tx + \thalf x^THx$, subproblem \eqref{g-dist} takes the form
\begin{equation}\label{s-dist}
  \min_{x\in X}\ J_0(x),\quad \mbox{where}\quad J_0(x) := \varphi(x)+\sum_{i\in\cI}\tdist{A_ix+b_i}{C_i}.
\end{equation}
Throughout our algorithm development and analysis, it is important to keep in mind that $\norm{y}\ne\tnorm{y}$ since we make heavy use of {\it both} of these norms.

\begin{example}[Intersections of Convex Sets]
  In many applications, the affine constraint has the representation $\hA x+\hb\in \hat{C}:=\bigcap_{i\in\cI} C_i$, where $C_i\subset\R^{m_i}$ is non-empty, closed, and convex for each $i\in\cI$.  Problems of this type are easily modeled in our framework by setting $A_i:=\hA$ and $b_i:=\hb$ for each $i\in\cI$, and $C:=C_1\times\cdots\times C_l$.
\end{example}

\subsection{Notation}

Much of the notation that we use is standard and based on that employed in \cite{RTRW}.  For convenience, we review some of this notation here.  The set $\Rn$ is the real $n$-dimensional Euclidean space with $\Rn_+$ being the positive orthant in $\Rn$ and $\Rn_{++}$ the interior of $\Rn_+$.  The set of real $m\times n$ matrices will be denoted as $\Rmn$.  The Euclidean norm on $\Rn$ is denoted $\tnorm{\cdot}$, and its closed unit ball is $\bB_2:=\set{x}{\tnorm{x}\le1}$.  The closed unit ball of the norm defined in \eqref{new norm} will be denoted by $\bB$.  Vectors in $\Rn$ will be considered as column vectors and so we can write the standard inner product on $\Rn$ as $\ip{u}{v}:=u^Tv$ for all $\{u,v\}\subset\Rn$.  The set $\bN$ is the set of natural numbers $\{1,2,\dots\}$. Given $\{u,v\}\subset\Rn$, the line segment connecting them is denoted by $[u,v]$. Given a set $X\subset\Rn$, we define the convex indicator for $X$ by
\[
  \indicator{x}{X}:=
  \begin{cases}
    0&\mbox{if $x\in X$,}\\
    +\infty&\mbox{if $x\notin X$,}
  \end{cases}
\]
and its support function by
\[
  \support{y}{X}:=\sup_{x\in X}\ip{y}{x}.
\]
A function $\map{f}{\Rn}{\eR:=\R\cup\{+\infty\}}$ is said to be convex if its epigraph,
\[
  \epi{f}:=\set{(x,\mu)}{f(x)\le\mu},
\]
is a convex set. The function $f$ is said to be closed (or lower semi-continuous) if $\epi{f}$ is closed, and $f$ is said to be proper if $f(x)>-\infty$ for all $x\in\Rn$ and $\dom{f}:=\set{x}{f(x)<\infty}\ne\emptyset$.  If $f$ is convex, then the subdifferential of $f$ at $\bx$ is given by 
\[
  \sd f(\bx):=\set{z}{f(\bx)+\ip{z}{x-\bx}\le f(x)\ \forall\, x\in\Rn}.
\]
Given a closed convex $X \subset \Rn$, the normal cone to $X$ at a point $\bx\in X$ is given by
\[
  \ncone{\bx}{X}:=\set{z}{\ip{z}{x-\bx}\le 0\ \forall\, x\in X}.
\]
It is well known that $\ncone{\bx}{X}=\sd \indicator{\bx}{X}$; e.g., see \cite{RTRW}.  Given a set $S\subset\Rm$ and a matrix $M\in\Rmn$, the inverse image of $S$ under $M$ is given by
\[
  M^{-1}S:=\set{x}{Mx\in S}.
\]

Since the set $C$ in \eqref{prod C} is non-empty, closed, and convex, the distance function $\dist{y}{C}$ is convex. Using the techniques of \cite{RTRW}, it is easily shown that the subdifferential of the distance function \eqref{prod dist} is
\begin{equation}\label{sd prod dist}
  \sd \dist{p}{C}=\sd \tdist{p_1}{C_1}\times\dots\times\sd\tdist{p_l}{C_l},
\end{equation}
where, for each $i\in\cI$, we have
\begin{equation}\label{sd prod dist detail}
  \sd\tdist{p_i}{C_i}=
  \begin{cases}
    \frac{(I - P_{C_i})p_i}{\tnorm{(I - P_{C_i})p_i}}&\mbox{if $i \not \in \cA(p)$,}\\
    \bB_2 \cap \ncone{p_i}{C_i}&\mbox{if $i \in \cA(p)$.}
  \end{cases}
\end{equation}
Here, we have defined 
\[
  \cA(p):=\set{i\in\cI}{\tdist{p_i}{C_i}=0}\ \ \forall\, p\in\Rm,
\]
and let $P_C(p)$ denote the projection of $p$ onto the set $C$ (see Theorem \ref{projection theorem}).

Since we will be working on the product space $\R^{m_1}\times\dots\times\R^{m_l}$, we will need notation for the components of the vectors in this space.  Given a vector $w\in \R^{m_1}\times\dots\times\R^{m_l}$, we denote the components in $\R^{m_i}$ by $w_i$ and the $j$th component of $w_i$ by $w_{ij}$ for $j=1,\dots,m_i$ and $i\in\cI$ so that $w=(w_1^T,\dots,w_l^T)^T$.  Correspondingly, given vectors $w_i\in\R^{m_i}$ for $i\in\cI$, we denote by $w\in\Rm$ the vector $w=(w_1^T,\dots,w_l^T)^T$.

\section{An Iterative Re-weighting Algorithm}\label{sec.irwa}

We now describe an iterative algorithm for minimizing the function $J_0$ in \eqref{s-dist}, where in each iteration one solves a subproblem whose objective is the sum of $\varphi$ and a weighted linear least-squares term.  An advantage of this approach is that the subproblems can be solved using matrix-free methods, e.g., the conjugate gradient (CG), projected gradient, and Lanczos \cite{GLRT99} methods.  The objectives of the subproblems are localized approximations to $J_0$ based on projections.  In this manner, we will make use of the following theorem.

\begin{theorem}\cite{Z71}\label{projection theorem}
  Let $C\subset\Rm$ be non-empty, closed, and convex.  Then, to every $y\in \Rm$, there is a unique $\by\in C$ such that
  \[
    \tnorm{y-\by}=\tdist{y}{C}.
  \]
  We call $\by = P_C(y)$ the projection of $y$ onto $C$.  Moreover, the following hold:
  \begin{enumerate}
    \item $\by=P_C(y)$ if and only if $\by\in C$ and $(y-\by)\in\ncone{\by}{C}$ \cite{RTR}.
    \item For all $\{y,z\}\subset\Rm$, the operator $P_C$ yields
    \[
      \tnorm{P_C(y) - P_C(z)}^2+\tnorm{(I-P_C)y - (I-P_C)z}^2\le \tnorm{y-z}^2.
    \]
  \end{enumerate}
\end{theorem}

Since $H$ is symmetric and positive semi-definite, there exists $A_0\in\R^{m_0\times n}$, where $m_0:=\text{rank}(H)$, such that $H=A_0^TA_0$.  We use this representation for $H$ in order to simplify our mathematical presentation; this factorization is not required in order to implement our methods.  Define $b_0:=0\in\Rn$, $C_0:=\{0\}\subset\Rn$, and $\cI_0 := \{0\} \cup \cI = \{0,1,\dots,l\}$.  Using this notation, we define our local approximation to $J_0$ at a given point $\tx$ and with a given \emph{relaxation vector} $\eps\in\R^l_{++}$ by
\[
  \hat G_{(\tx,\eps)}(x) := g^Tx + \thalf \sum_{i\in\cI_0} w_i(\tx,\eps)\tnorm{A_i x +b_i - \Pri(A_i\tx +b_i)}^2,
\]
where, for any $x\in\Rn$, we define
\begin{align}\label{w}
  w_0(x,\eps) :=1,\quad
  w_i(x,\eps)&:=\left(\mathrm{dist}_2^2(A_ix+b_i\mid C_i)+\eps_i^2\right)^{-1/2}\quad \forall\, i\in\cI,\\ \nonumber
  \mbox{and}\quad W(x,\eps)  &:=\mathrm{diag}(w_0(x,\eps)I_{m_0},\dots,w_l(x,\eps) I_{m_l}).
\end{align}
Define
\begin{equation}\label{tA}
  \tA:= \begin{bmatrix} A_0 \\ A \end{bmatrix}.
\end{equation}

We now state the algorithm.
\medskip

\underline{\bf{Iterative Re-Weighting Algorithm (IRWA)}}
\begin{enumerate}
  \item[Step 0:] (Initialization)
  Choose an initial point $x^0\in X$, an initial relaxation vector $\eps^0 \in\bR^l_{++}$, and scaling parameters $\eta\in(0,1)$, $\gamma >0$, and $M>0$.  Let $\sigma\geq0$ and $\sigma'\geq0$ be two scalars which serve as termination tolerances for the stepsize and relaxation parameter, respectively.  Set $k:=0$.
  \item[Step 1:] (Solve the re-weighted subproblem for $x^{k+1}$) \\
  \noindent
  Compute a solution $x^{k+1}$ to the problem
  \begin{equation}\label{Gcal}
    \cG{(x^k,\eps^k)}:\quad \min_{x\in X} \hat G_{(x^k,\eps^k)}(x).
  \end{equation}
  \item[Step 2:] (Set the new relaxation vector $\eps^{k+1}$)\\
  \noindent
  Set 
  \[
    q_i^k:=A_i (x^{k+1} - x^k)\ \ \mbox{and}\ \ r_i^k:=(I-\Pri)(A_i x^k + b_i)\quad \forall\, i\in\cI_0. 
  \]
  If  
  \begin{equation}\label{step 2 inequality}
    \tnorm{q_i^k} \leq  M\Big[ \tnorm{r_i^k}^2 + (\eps_i^k)^2\Big]^{\frac{1}{2}+\gamma}\quad \forall\, i \in \cI, 
  \end{equation}
  then choose $\eps^{k+1} \in(0,\ \eta \eps^k]$; else, set $\eps^{k+1} := \eps^k$. 
  \item[Step 3:] (Check stopping criteria)\\
  \noindent
  If $\norm{x^{k+1} - x^k}_2 \leq \sigma$ and $\norm{\eps^{k}}_2 \leq \sigma'$, then stop; else, set $k := k+1$ and go to Step 1. 
\end{enumerate}
\medskip

\begin{rem}
  In cases where $C_i=\{0\}\subset\R$ for all $i\in\cI$ and $\phi\equiv 0$, this algorithm has a long history in the literature. Two early references are \cite{BeTu74} and \cite{Schlos73}. In such cases, the algorithm reduces to the classical algorithm for minimizing $\onorm{Ax+b}$ using iteratively re-weighted least-squares.
\end{rem}
\begin{rem}
  If there exists $z_0$ such that $A_0^Tz_0=g$, then, by setting $b_0:=z_0$, the linear term $g^Tx$ 
  can be eliminated in the definition of $\hG$.
\end{rem}
\begin{rem}\label{rem duality gap}
  It is often advantageous to employ a stopping criteria based on a percent reduction in the duality gap rather than the stopping criteria given in Step 3 above \cite{Burke85,Bur87b}.  
  In such cases, one keeps track of both the primal objective values $J_0^k:=J_0(x^k)$ and the dual objective values
  \[
    \hJ_0^k:=\thalf(g+A^T\tu^k)^TH^{-1}(g+A^T\tu^k)-b^T\tu^k+\sum_{i\in\cI}\support{\tu^k_i}{C_i},
  \]
  where the vectors $\tu^k:=W_kr^k$ are dual feasible (see \eqref{dual problem} for a discussion of the dual problem).  Given $\sigma\in(0,1)$, Step 3 above can be replaced by
  \begin{enumerate}
    \item[Step 3':] (Check stopping criteria)\\
    If $(J_0^1+\hJ_0^k) \leq \sigma(J_0^1-J_0^k)$, then stop; else, set $k := k+1$ and go to Step 1.
  \end{enumerate}
  \noindent
  This is the stopping criteria employed in some of our numerical experiments.  Nonetheless, for our analysis, we employ Step 3 as it is stated in the formal description of IRWA for those instances when dual values $\hJ_0^k$ are unavailable, such as when these computations are costly or subject to error.
\end{rem}

\subsection{Smooth approximation to $J_0$}

Our analysis of IRWA is based on a smooth approximation to $J_0$.  Given $\eps\in\R^l_+$, define the $\eps$-smoothing of $J_0$ by
\begin{equation}\label{J eps}
  J(x,\eps):=\varphi(x)+\sum_{i\in\cI}\sqrt{\mathrm{dist}_2^2(A_ix+b_i\mid C_i)+\eps_i^2}\ .
\end{equation}
Note that $J_0(x) \equiv J(x,0)$ and that $J(x,\eps)$ is jointly convex in $(x,\eps)$ since
\[
  J(x,\eps)=\varphi(x)+\sum_{i\in\cI}\tdist{\begin{bmatrix}A_i& 0\\ 0&e_i^T\end{bmatrix} \begin{pmatrix}x\\ \eps\end{pmatrix}+\begin{pmatrix}b_i\\     0\end{pmatrix}}{C_i\times \{0\}},
\]
where $e_i$ is the $i$th unit coordinate vector.  By \cite[Corollary 10.11]{RW98}, \eqref{sd prod dist}, and \eqref{sd prod dist detail},
\begin{gather}\label{partial sd for J}
  \sd J_0(x)=\sd_x J(x,0)=\nabla \varphi(x)+ A^T\sd \dist{\cdot}{C}(Ax+b)= \\
  \nabla \varphi(x)+\sum_{i \not \in \mathcal{A}(Ax+b)} A_i^T \frac{(I - P_{C_i})(A_i x + b_i)}{\tnorm{(I - P_{C_i})(A_i x + b_i)}} + \sum_{i \in \mathcal{A}(Ax+b)} A_i^T ( \bB_2 \cap \ncone{A_i x +b_i}{C_i}). \nonumber
\end{gather}
Given $\tx\in\Rn$ and $\teps\in\R^l_{++}$, we define a weighted approximation to $J(\cdot,\teps)$ at $\tx$ by
\[
  G_{(\tx,\teps)}(x) := g^Tx + \thalf \sum_{i\in\cI_0} w_i(\tx,\teps)\sdist{A_ix +b_i}{C_i}.
\]

We have the following fundamental fact about solutions of $\cG(\tx,\teps)$ defined by \eqref{Gcal}.

\begin{lemma}\label{Gfun}
  Let $\tx\in X$, $\teps\in\R^l_{++}$, $\heps \in (0,\teps]$, and $\hx \in \argmin_{x\in X} \hat G_{(\tx,\teps)}(x)$.  Set
$\wri := w_i(\tx,\teps)$ and $\qri:=A_i(\hx-\tx)$ for $i\in\cI_0$, $\tW:=W(\tx,\teps)$, and $q:=(q_0^T,\dots,q_l^T)^T$.  Then,
  \begin{equation}\label{G descent}
    G_{(\tx,\teps)}(\hx) -  G_{(\tx ,\teps)}(\tx ) \leq -\thalf q^T\tW q
\end{equation}
and
\begin{equation}\label{J descent}
  J(\tx,\teps) - J(\hx,\heps) \geq \thalf q^T \tW q.
\end{equation}
\end{lemma}
\begin{proof}
  We first prove \eqref{G descent}.  Define $\hr_i:=(I- P_{C_i})(A_i \hx + b_i)$ and $\tr_i:=(I- P_{C_i})(A_i \tx + b_i)$ for $i\in\cI_0$, and set $\hr:=(\hr_0^T,\dots,\hr_l^T)^T$ and $\tr:=(\tr_0^T,\dots,\tr_l^T)^T$.  Since $\hx \in \argmin_{x\in X} \hat G_{(\tx,\teps)}(x)$, there exists $\hv\in\ncone{\hx}{X}$ such that
  \[
    0 = g + \tA^T\tW(\tA\hx+b-P_C(\tA\tx+b))+\hv=g+\tA^T\tW(q+\tr)+\hv,
  \]
  or, equivalently,
  \begin{equation}\label{normal eq}
    -\hv=g+\tA^T\tW(q+\tr).
  \end{equation}
  Moreover, by the definition of the projection operator $\Pri$, we know that
  \[
  \tnorm{\hr_i}=\tnorm{(I - P_{C_i})(A_i \hx  + b_i)} \le \tnorm{A_i\hx  + b_i - P_{C_i}(A_i \tx + b_i)} = \tnorm{q_i+\tr_i}
  \]
  so that
  \begin{equation}\label{proj 1}
    \tnorm{\hr_i}^2-\tnorm{q_i+\tr_i}^2\le 0\quad \forall\, i \in \cI_0.
  \end{equation}
  Therefore,
  \[
  \begin{aligned}
     &\ G_{(\tx ,\teps)}(\hx ) -  G_{(\tx ,\teps)}(\tx )  \\
    =&\ g^T(\hx - \tx ) + \thalf \textstyle\sum_{i\in\cI_0} \tw_i [\tnorm{\hr_i}^2 - \tnorm{\tr_i}^2 ] \\
    =&\ g^T(\hx - \tx ) + \thalf \textstyle\sum_{i\in\cI_0} \tw_i [(\tnorm{\hr_i}^2  - \tnorm{q_i+\tr_i}^2) + (\tnorm{q_i+\tr_i}^2  - \tnorm{\tr_i}^2) ] \\
  \le&\ g^T(\hx - \tx ) + \thalf \textstyle\sum_{i\in\cI_0} \tw_i [\tnorm{q_i+\tr_i}^2  - \tnorm{\tr_i}^2]
     &&(\mbox{by \eqref{proj 1}})\\
    =&\ g^T(\hx - \tx ) + \thalf \textstyle\sum_{i\in\cI_0} \tw_i [\tnorm{q_i}^2+2\ip{q_i}{\tr_i}]\\
    =&\ g^T(\hx - \tx ) + \thalf \textstyle\sum_{i\in\cI_0} \tw_i [-\tnorm{q_i}^2+2\ip{q_i}{q_i+\tr_i}]\\
    =&\ -\thalf q^T\tW q+g^T(\hx  - \tx )+q^T\tW(q+\tr)\\
    =&\ -\thalf q^T\tW q+(\hx-\tx)^T(g+\tA^T\tW(q+\tr))\\
    =&\ -\thalf q^T\tW q+(\tx-\hx)^T\hv
     &&(\mbox{by \eqref{normal eq}})\\
  \le&\ - \thalf q^T\tW q,
  \end{aligned}
  \]
  where the final inequality follows since $\tx\in X$ and $\hv\in\ncone{\hx}{X}$.
 
  We now prove \eqref{J descent}.  Since $\sqrt{t}$ is a concave function of $t$ on $\R_+$, we have
  \[
    \sqrt{\hat{t}} \leq \sqrt{\tilde{t}} + \frac{\hat{t} - \tilde{t}}{2\sqrt{\tilde{t}}}\quad \forall\, \{\hat{t},\tilde{t}\} \subset \R_{++},
  \]
  and so, for $i\in\cI$, we have
  \begin{align}
       &\ \sqrt{\mathrm{dist}_2^2(A_i\hx+b_i\mid C_i)+\teps_i^2} \nonumber \\
    \le&\ \sqrt{\sdist{A_i\tx +b_i}{C_i} + \teps_i^2}+\frac{\sdist{A_i\hx + b_i }{C_i} -\sdist{A_i\tx +b_i}{C_i}}{2\sqrt{\sdist{A_i\tx +b_i}{C_i} + \teps_i^2}}. \label{ccv}
  \end{align}
  Hence,
  \begin{align*}
    J(\hx,\heps)
      &\le J(\hx,\teps) =\varphi(\hx)+\sum_{i\in\cI}\sqrt{\mathrm{dist}_2^2(A_i\hx+b_i\mid C_i)+\teps_i^2} \\
      &\le J(\tx,\teps) +(\varphi(\hx)-\varphi(\tx))+ \thalf \sum_{i\in\cI} \frac{\sdist{A_i\hx + b_i }{C_i} -\sdist{A_i\tx +b_i}{C_i}}{\sqrt{\sdist{A_i\tx +b_i}{C_i} + \teps_i^2}} \\
      &=   J(\tx,\teps) + [G_{(\tx,\teps)}(\hx) -G_{(\tx,\teps)}(\tx) ]\\
      &\le J(\tx,\teps) - \thalf q^T \tW q,
  \end{align*}
  where the first inequality follows from $\heps \in (0,\teps]$, the second inequality follows from \eqref{ccv}, and the third inequality follows from \eqref{G descent}.
\end{proof}

\subsection{Coercivity of $J$}

Lemma \ref{Gfun} tells us that IRWA is a descent method for the function $J$.  Consequently, both the existence of solutions to \eqref{s-dist} as well as the existence of cluster points to IRWA can be guaranteed by understanding conditions under which the function $J$ is coercive, or equivalently, conditions that guarantee the boundedness of the lower level sets of $J$ over $X$. For this, we need to consider the asymptotic geometry of $J$ and $X$.

\begin{defn}\cite[Definition 3.3]{RW98}
  Given $Y\subset \Rm$, the horizon cone of $Y$ is
  \[
    Y^\infty:=\set{z}{\exists\, t^k\downarrow 0,\ \{y^k\}\subset Y\mbox{ such that }t^ky^k\rightarrow z}.
  \]
\end{defn}

We have the basic facts about horizon cones given in the following proposition.

\begin{prop}\label{hzn basic facts}
  The following hold:
  \begin{enumerate}
    \item The set $Y\subset \Rm$ is bounded if and only if $Y^\infty=\{0\}$.
    \item Given $Y_i\subset\R^{m_i}$ for $i\in\cI$, we have $(Y_1\times \cdots\times Y_l)^\infty=Y_1^\infty\times \cdots\times Y_l^\infty$.
    \item \cite[Theorem 3.6]{RW98}\label{rec C} If $C\subset\Rm$ is non-empty, closed, and convex, then
    \[
      C^\infty=\set{z}{C+z\subset C}.
    \]
  \end{enumerate}
\end{prop}

We now prove the following result about the lower level sets of $J$.

\begin{theorem}\label{thm:boundedness}
  Let $\alf>0$ and $\eps\in\R^l_+$ be such that the set
  \[
    L(\alf,\eps):=\set{x\in X}{J(x,\eps)\le\alf}
  \] 
  is non-empty.  Then,
  \begin{equation}\label{L hzn}
    L(\alf,\eps)^\infty=\set{\bx\in X^\infty}{g^T\bx\le 0,\ H\bx=0,\ A\bx\in C^\infty}.
  \end{equation}
  Moreover, $L(\alf,\eps)$ is compact for all $(\alf,\eps)\in\R^{l+1}_+$ if and only if
  \begin{equation}\label{boundedness}
    \left[\bx\in X^\infty\cap\ker(H)\cap A^{-1}C^\infty\mbox{ satisfies }g^T\bx\le 0\right]\iff \bx=0.
  \end{equation}
\end{theorem}
\begin{proof}
  Let $x\in L(\alf,\eps)$ and let $\bx$ be an element of the set on the right-hand side of \eqref{L hzn}.  Then, by Proposition \ref{hzn basic facts}, for all $\lam \ge 0$ we have $x+\lam \bx\in X$ and $\lam A_i\bx+C_i\subset C_i$ for all $i\in\cI$, and so for each $i\in\cI$ we have
  \begin{eqnarray*}
    \dist{A_i(x+\lam \bx)+b_i}{C_i}
      &\le&\dist{A_i(x+\lam \bx)+b_i}{\lam A_i\bx+C_i}\\
      &=  &\dist{(A_ix+b_i)+\lam A_i\bx}{\lam A_i\bx+C_i}\\
      &=  &\dist{A_ix+b_i}{C_i}.
  \end{eqnarray*}
  Therefore, 
  \begin{eqnarray*}
    J(x+\lam\bx,\eps)
      &=  &\varphi(x)+\lam g^T\bx+\sum_{i\in\cI} \sqrt{\sdist{A_i(x+\lam \bx)+b_i}{C_i}+\eps_i^2} \\ 
      &\le&\varphi(x)+\sum_{i\in\cI} \sqrt{\sdist{A_ix+b_i}{C_i}+\eps_i^2} = J(x,\eps) \le \alf.
  \end{eqnarray*}
  Consequently, $\bx\in L(\alf,\eps)^\infty$.

  On the other hand, let $\bx\in L(\alf,\eps)^\infty$.  We need to show that $\bx$ is an element of the set on the right-hand side of \eqref{L hzn}. For this, we may as well assume that $\bx\ne 0$. By the fact that $\bx \in L(\alf,\eps)^\infty$, there exists $t^k\downarrow 0$ and $\{x^k\}\subset X$ such that $J(x^k,\eps)\le \alf$ and $t^kx^k\rightarrow \bx$.  Consequently, $\bx\in X^\infty$.  Moreover,
  \[
    g^T(t^kx^k)=t^k(g^Tx^k)\le t^kJ(x^k,\eps)\le t^k\alf\rightarrow 0
  \]
  and so
  \begin{align*}
    0 & \le \norm{A_0(t^kx^k)}^2=(t^kx^k)^TH(t^kx^k)=(t^k)^2(x^k)^THx^k \\
      & \le(t^k)^2 2(J(x^k,\eps) - g^Tx^k)\le (t^k)^2 2\alf - t^k 2 g^T(t^kx^k)\rightarrow 0.
  \end{align*}
  Therefore, $g^T\bx\le 0$ and $H\bx =0$.  Now, define $z^k:=P_C(Ax^k+b)$ for $k\in\bN$.  Then, by Theorem~\ref{projection theorem}(2), we have
  \[
    \tnorm{z^k}\le \tnorm{(I-P_C)(Ax^k+b)}+\tnorm{Ax^k+b}\le \alf +\tnorm{Ax^k+b},
  \]
  which, since ${A(t^kx^k)+t^kb}\rightarrow A\bx$, implies that the sequence $\{t^kz^k\}$ is bounded.  Hence, without loss of generality, we can assume that there is a vector $\bz$ such that $t^kz^k\rightarrow \bz$, where by the definition of $z^k$ we have $\bz\in C^\infty$.  But,
  \begin{align*}
    0 \leq \tnorm{A(t^kx^k)+t^kb-(t^kz^k)}
      & =   t^k\tdist{Ax^k+b}{C} \\
      &\le t^kJ(x^k,\eps) - t^k g^Tx^k \le t^k\alf - g^T(t^kx^k)\rightarrow 0,
  \end{align*}
  while
  \[
    \tnorm{A(t^kx^k)+b-(t^kz^k)}\rightarrow\tnorm{A\bx-\bz}.
  \]
  Consequently, $\bx\in X^\infty,\ g^T\bx\le 0,\ H\bx=0$, and $A\bx\in C^\infty$, which together imply that $\bx$ is in the set on the right-hand side of \eqref{L hzn}.
\end{proof}

\begin{cor}\label{cor:compactness}
  Suppose that the sequence $\{(x^k,\eps^k)\}$ is generated by IRWA with initial point $x^0\in X$ and relaxation vector $\epsilon^0 \in \bR^l_{++}$.  Then, $\{x^k\}$ is bounded if \eqref{boundedness} is satisfied, which follows if at least one of the following conditions holds:
  \begin{enumerate}
    \item $X$ is compact.
    \item $H$ is positive definite.
    \item $C$ is compact and $X^\infty\cap \ker(H)\cap\ker(A)=\{0\}$.
  \end{enumerate}
\end{cor}

\begin{rem}
  For future reference, observe that
  \begin{equation}\label{kernel tA}
    \ker(H)\cap\ker(A)=\ker(\tA),
  \end{equation}
  where $\tA$ is defined in \eqref{tA}.
\end{rem}

\subsection{Convergence of IRWA}

We now return to our analysis of the convergence of IRWA by first proving the following lemma that discusses critical properties of the sequence of iterates computed in the algorithm.

\begin{lemma}\label{irwa properties}
  Suppose that the sequence $\{(x^k,\eps^k)\}$ is generated by IRWA with initial point $x^0 \in X$ and relaxation vector $\epsilon^0 \in \bR^l_{++}$, and, for $k\in\bN$, let $q_i^k$ and $r_i^k$ for $i\in\cI_0$ be as defined in Step 2 of the algorithm with 
  \[
    q^k:=((q_0^k)^T,\dots,(q_l^k)^T)^T
    \quad\mbox{and}\quad
    r^k:=((r_0^k)^T,\dots,(r_l^k)^T)^T.
  \]
  Moreover, for $k \in \bN$, define 
  \[
    w^k_i:=w_i(x^k,\eps^k)\ \mbox{for}\ i\in\cI_0
    \quad\mbox{and}\quad
    W_k:=W(x^k,\eps^k),
  \]
  and set $S:=\set{k}{\eps^{k+1}\le\eta\eps^k}$.  Then, the sequence $\{J(x^{k}, \eps^k)\}$ is monotonically decreasing.  Moreover, either $\inf_{k\in\bN} J(x^k,\eps^k) = -\infty$, in which case $\inf_{x\in X} J_0(x) = -\infty$, or the following hold: 
  \begin{enumerate}
    \item $\sum_{k=0}^\infty (q^k)^TW_kq^k< \infty$.
    \item $\eps^k\rightarrow 0$ and $H(x^{k+1}-x^k)\to 0$.
    \item $W_kq^k\overset{S}{\rightarrow}0$.
    \item $w_i^kr_i^k=r^k_i/\sqrt{\tnorm{r^k_i}^2+\eps^k_i}\in\bB_2\cap\ncone{\Pri(A_ix^k+b_i)}{C_i},\, i\in\cI,\, k\in\bN$.
    \item $-\tA^TW_kq^k\in (\nabla\varphi(x^k)+\sum_{i\in\cI} A_i^Tw_i^kr_i^k) + \ncone{x^{k+1}}{X},\, k\in\bN$.
    \item If $\{\dist{Ax^k+b}{C}\}_{k\in S}$ is bounded, then $q^k\overset{S}{\rightarrow}0$. 
  \end{enumerate}
\end{lemma}
\begin{proof}
  The fact that $\{J(x^{k}, \eps^k)\}$ is monotonically decreasing is an immediate consequence of the monotonicity of the sequence $\{\epsilon_k\}$, Lemma \ref{Gfun}, and the fact that $W_k$ is positive definite for all $k \in \bN$.  If $J(x^{k}, \eps^k)\to -\infty$, then $\inf_{x\in X}J_0(x)\ =-\infty$ since $J_0(x)=J(x,0)\le J(x,\eps)$ for all $x\in\Rn$ and $\eps\in\R^l_+$.  All that remains is to show that Parts (1)--(6) hold when $\inf_{k\in\bN} J(x^k,\eps^k)>-\infty$, in which case we may assume that the sequence $\{J(x^k,\eps^k)\}$ is bounded below.  We define the lower bound $\tJ:=\inf_{k\in\bN} J(x^k,\eps^k) = \lim_{k\in\bN} J(x^k,\eps^k)$ for the remainder of the proof.

  \noindent
  (1) By Lemma \ref{Gfun}, for every positive integer $\bk$ we have
  \begin{align*}
    \thalf \sum_{k=0}^\bk (q^k)^TW_kq^k 
      &\leq \sum_{k=0}^\bk [J(x^{k}, \eps^k) - J(x^{k+1}, \eps^{k+1})] \\
      &=    J(x^{0}, \eps^0) - J(x^{\bk+1}, \eps^{\bk+1}) \\
      &\leq J(x^{0}, \eps^0) - \tJ.
  \end{align*}
  Therefore, as desired, we have
  \begin{equation*}
    \sum_{k=0}^\infty (q^k)^TW_kq^k \leq 2(J(x^{0}, \eps^0)-\tJ) < \infty.
  \end{equation*}

  \noindent
  (2) Since $\eta\in (0,1)$, if $\eps^k \nrightarrow 0$, then there exists an integer $\bk \geq 0$ and a scalar $\beps > 0$ such that $\eps^k = \beps$ for all $k\ge\bk$.   Part (1) implies that $(q^k)^TW_kq^k$ is summable so that $(w_i^k\tnorm{q^k_i})(\tnorm{q^k_i})=w_i^k\tnorm{q^k_i}^2 \rightarrow 0$ for each $i\in\cI_0$.  In particular, since $w_0^k := 1$ for all $k \in \bN$, this implies that $q^k_0\to 0$, or equivalently that $H(x^{k+1}-x^k)\to 0$.  In addition, since for each $i\in\cI$ both sequences $\{\tnorm{q^k_i}\}$ and $\{w^k_i\tnorm{q^k_i}\}$ cannot be bounded away from $0$, there is a subsequence $\hS\subset\bN$ and a partition $\{\cI_1,\cI_2\}$ of $\cI$ such that $\tnorm{q^k_i}\overset{\hS}{\rightarrow} 0$ for all $i\in \cI_1$ and $w_i^k\tnorm{q^k_i}\overset{\hS}{\to}0$ for all $i\in \cI_2$.  Hence, there exists $k_0\in \hS$ such that for all $k\ge k_0$ we have
  \begin{equation*}
    \begin{array}{rrcll}
      & \tnorm{q^k_i}      &\le& M\Big[ \tnorm{r_i^k}^2 +\beps_i^2 \Big]^{\half+\gamma}&\forall\, i \in \cI_1 \\
      \mbox{and}
      & w_i^k\tnorm{q^k_i} &\le& M\Big[ \tnorm{r_i^k}^2 +\beps_i^2 \Big]^{\gamma}&\forall\, i \in \cI_2.
    \end{array}
  \end{equation*}
  Therefore, since $w_i^k = (\|r_i^k\|^2 + (\eps_i^k)^2)^{-1/2}$, we have for all $k_0\le k\in \hS$ that
  \[
    \tnorm{q_i^k} \leq  M\Big[ \tnorm{r_i^k}^2 + \beps_i^2 \Big]^{\frac{1}{2}+\gamma}\quad \forall i \in \cI.
  \]
  However, for every such $k$, Step 2 of the algorithm chooses $\eps^{k+1}\in(0,\eta\eps^k]$.  This contradicts the supposition that $\eps^k = \beps > 0$ for all $k\ge\bk$, so we conclude that $\eps^k \rightarrow 0$.
 
  \noindent
  (3) It has just been shown in Part (2) that $w^k_0q^k_0=q^k_0\to 0$, so we need only show that $w^k_i\tnorm{q^k_i}\overset{S}{\to} 0$ for each $i \in \cI$.  
  
  Our first step is to show that for every subsequence $\hS\subset S$ and $i_0\in\cI$, there is a further subsequence $\tS\subset \hS$ such that $w^k_{i_0}\tnorm{q^k_{i_0}}\overset{\tS}{\to} 0$. The proof uses a trick from the proof of Part (2).  Let $\hS\subset S$ be a subsequence and $i_0\in\cI$.  Part~(1) implies that $(w_i^k\tnorm{q^k_i})(\tnorm{q^k_i})=w_i^k\tnorm{q^k_i}^2 \rightarrow 0$ for each $i\in\cI_0$.  As in the proof of Part (2), this implies that there is a further subsequence $\tS\subset \hS$ and a partition $\{\cI_1,\cI_2\}$ of $\cI$ such that $\tnorm{q^k_i}\overset{\tS}{\rightarrow} 0$ for all $i\in \cI_1$ and $w_i^k\tnorm{q^k_i}\overset{\tS}{\to}0$ for all $i\in \cI_2$.  If $i_0\in \cI_2$, then we would be done, so let us assume that $i_0\in \cI_1$.  We can assume that $\tS$ contains no subsequence on which $w^k_{i_0}\tnorm{q_{i_0}^k}$ converges to $0$ since, otherwise, again we would be done.  Hence, we assume that $w^k_{i_0}\tnorm{q_{i_0}^k}\overset{\tS}{\nrightarrow} 0$.  Since $\tnorm{q^k_{i_0}}\overset{\tS}{\rightarrow} 0$ as $i_0\in \cI_1$, this implies that there is a subsequence $\tS_0\subset\tS$ such that $w^k_{i_0}\overset{\tS_0}{\to}\infty$, i.e., $(\tnorm{r_{i_0}^k}^2 +(\eps^k_{i_0})^2)\overset{\tS_0}{\to}0$.  But, by Step 2 of the algorithm, for all $k\in S$,
  \[
    \tnorm{q_{i}^k} \leq  M\Big[ \tnorm{r_{i}^k}^2 + (\eps_{i}^k)^2 \Big]^{\frac{1}{2}+\gamma}\quad \forall\, i \in \cI,
  \]
  or, equivalently,
  \[
    w^k_{i}\tnorm{q_{i}^k} \leq  M\Big[ \tnorm{r_{i}^k}^2 + (\eps_{i}^k)^2 \Big]^{\gamma}\quad \forall\, i \in \cI,
  \]
  giving the contradiction $w^k_{i_0}\tnorm{q_{i_0}^k}\overset{\tS_0}{\to}0$.  Hence, $w^k_{i_0}\tnorm{q_{i_0}^k}\overset{\tS}{\to}0$, and we have shown that for every subsequence $\hS\subset S$ and $i_0\in\cI$, there is $\tS\subset \hS$ such that $w^k_i\tnorm{q^k_i}\overset{\tS}{\to} 0$.

  Now, if $W_kq^k\overset{S}{\nrightarrow}0$, then there would exist a subsequence $\hS\subset S$ and an index $i\in\cI$ such that $\{w^k_i\tnorm{q^k_i}\}_{k\in\hS}$ remains bounded away from $0$.  But, by what we have just shown in the previous paragraph, $\hS$ contains a further subsequence $\tS\subset\hS$ with $w^k_i\tnorm{q^k_i}\overset{\tS}{\to}0$. This contradiction establishes the result.
 
  \noindent
  (4) By Theorem \ref{projection theorem}, we have
  \[
    r^k_i\in\ncone{\Pri(A_ix^k+b_i)}{C_i}\quad\forall\, i\in\cI_0,\ k\in\bN,
  \]
  from which the result follows.
 
  \noindent
  (5) By convexity, the condition $x^{k+1}\in\argmin_{x\in X} \hG_{(x^k,\eps^k)}(x)$ is equivalent to
  \begin{eqnarray*}
    0 &\in& \nabla_x\hG_{(x^k,\eps^k)}(x^{k+1})+\ncone{x^{k+1}}{X} \\
      &=  & g+\sum_{i\in\cI_0}A_i^Tw_i^k(q^k_i + r_i^k) + \ncone{x^{k+1}}{X}\\
      &=  &\tA^TW_kq^k+\nabla\varphi(x^k)+\sum_{i\in\cI}A_i^Tw_i^kr_i^k + \ncone{x^{k+1}}{X}.
  \end{eqnarray*}
  
  \noindent
  (6) Let $i\in\cI$.  We know from Part (3) that $w^k_i\tnorm{q^k_i}\overset{S}{\to} 0$.  If $\tnorm{q^k_i}\overset{S}{\nrightarrow} 0$, then there exists a subsequence $\hS\subset S$ such that $\{\tnorm{q^k_i}\}_{k\in\hS}$ is bounded away from $0$, which would imply that $(\tnorm{r_i^k}^2 + (\eps^k_i)^2)^{-1/2}=w^k_i\overset{\hS}{\to} 0$.  But then $\tnorm{r^k_i}\overset{\hS}{\to}\infty$ since $0\le \eps^k\le\eps^0$, which contradicts the boundedness of $\{\dist{Ax^k+b}{C}\}_{k\in S}$.
\end{proof}

In the next result, we give conditions under which every cluster point of the subsequence $\{x^k\}_{k\in S}$ is a solution to $\min_{x\in X} J_0(x)$, where $S$ is defined in Lemma~\ref{irwa properties}.  Since $J_0$ is convex, this is equivalent to showing that $0\in \sd J_0(\bx)+\ncone{\bx}{X}.$

\begin{theorem}\label{main}
  Suppose that the sequence $\{(x^k,\eps^k)\}$ is generated by IRWA with initial point $x^0\in X$ and relaxation vector $\epsilon^0 \in \bR^l_{++}$, and that the sequence $\{J(x^k,\eps^k)\}$ is bounded below.  Let $S$ be defined as in Lemma \ref{irwa properties}.  If either
  \begin{enumerate}
    \item[(a)] $\ker(A)\cap\ker(H)=\{0\}$ and $\{\dist{Ax^k+b}{C}\}_{k\in S}$ is bounded, or
    \item[(b)] $X=\Rn$,
  \end{enumerate} 
  then any cluster point $\bx$ of the subsequence $\{x^k\}_{k \in S}$ satisfies $0\in\sd J_0(\bx)+\ncone{\bx}{X}$.  Moreover, if (a) holds, then $(x^{k+1}-x^k)\overset{S}{\to} 0$.
\end{theorem}
\begin{proof}
  Let the sequences $\{q^k\},\ \{r^k\}$ and $\{W_k\}$ be defined as in Lemma \ref{irwa properties}, and let $\bx$ be a cluster point of the subsequence $\{x^k\}_{k \in S}$.  Let $\hS\subset S$ be a subsequence such that $x^k \overset{\hS}{\rightarrow} \bar x$.  Without loss of generality, due to the upper semi-continuity of the normal cone operator, the continuity of the projection operator and Lemma \ref{irwa properties}(4), we can assume that for each $i\in \mathcal{A}(A\bx+b)$ there exists
  \begin{equation}\label{limit not in I}
    \bu_i\in\bB_2\cap\ncone{A_i\bx+b_i}{C_i}
    \quad\mbox{such that}\quad
    w^k_ir^k_i\overset{\hS}{\rightarrow} \bu_i.
  \end{equation}
  Also due to the continuity of the projection operator, for each $i\notin I(A\bx+b)$ we have 
  \begin{equation}\label{limit in I}
    w^k_ir^k_i\overset{\hS}{\rightarrow} \frac{(I-P_{C_i})(A_i\bx+b_i)}{\tnorm{(I-P_{C_i})(A_i\bx+b_i)}}.
  \end{equation}

  Let us first suppose that (b) holds, i.e., that $X=\Rn$ so that $\ncone{x}{X}=\{0\}$ for all $x\in\Rn$.  By \eqref{limit not in I}-\eqref{limit in I}, Lemma~\ref{irwa properties} Parts (3) and (5), and \eqref{partial sd for J}, we have
  \[
    \begin{aligned}
      0 &\in \nabla \varphi(\bx)+\!\!\!\!\! \sum_{i\notin \mathcal{A}(A\bx+b)}A_i^T\frac{(I-P_{C_i})(A_i\bx+b_i)}{\tnorm{(I-P_{C_i})(A_i\bx+b_i)}}
+\!\!\!\!\! \sum_{i\in \mathcal{A}(A\bx+b)}A_i^T(\bB_2\cap\ncone{A_i\bx+b_i}{C_i}) \\ 
        &=   \sd J_0(\bx).
    \end{aligned}
  \]

  Next, suppose that (a) holds, i.e., that $\ker(A)\cap\ker(H)=\{0\}$ and the set $\{\dist{Ax^k+b}{C}\}_{k\in S}$ is bounded.  This latter fact and  Lemma~\ref{irwa properties}(6) implies that $q^k\overset{S}{\to}0$.  We now show that $(x^{k+1}-x^k)\overset{S}{\to} 0$.  Indeed, if this were not the case, then there would exist a subsequence $\hS\subset S$ and a vector $\bw\in\Rn$ with $\tnorm{\bw}=1$ such that $\{\tnorm{x^{k+1}-x^k}\}_{\hS}$ is bounded away from $0$ while $\frac{x^{k+1}-x^k}{\tnorm{x^{k+1}-x^k}}\overset{\hS}{\to} \bw$.  But then $q^k/\tnorm{x^{k+1}-x^k}\overset{\hS}{\to} 0$ while $q^k/\tnorm{x^{k+1}-x^k}=\tA\frac{x^{k+1}-x^k}{\tnorm{x^{k+1}-x^k}}\overset{\hS}{\to}\tA\bw$, where $\tA$ is defined in \eqref{tA}.  But then $0\ne \bw\in\ker(H)\cap\ker(A) = \ker(\tA)$, a contradiction. Hence, $(x^{k+1}-x^k)\overset{S}{\to} 0$, and so $x^{k+1}=x^k+(x^{k+1}-x^k)\overset{S}{\to} \bx$.  In particular, this and the upper semi-continuity of the normal cone operator imply that $\limsup_{k\in S}\ncone{x^{k+1}}{X}\subset\ncone{\bx}{X}$.  Hence, by \eqref{limit not in I}--\eqref{limit in I}, Lemma \ref{irwa properties} Parts (3) and (5), and \eqref{partial sd for J}, we have
  \[
    \begin{aligned}
      0 &\in \nabla \varphi(\bx)+\!\!\!\!\! \sum_{i\notin I(A\bx+b)}A_i^T\frac{(I-P_{C_i})(A_i\bx+b_i)}{\tnorm{(I-P_{C_i})(A_i\bx+b_i)}} + \!\!\!\!\! \sum_{i\in I(A\bx+b)}A_i^T(\bB_2\cap\ncone{A_i\bx+b_i}{C_i}) \\ 
        &\qquad\qquad + \ncone{\bx}{X} \\ 
        &= \sd J_0(\bx)+\ncone{\bx}{X},
    \end{aligned}
  \]
  as desired.
\end{proof}
 
The previously stated Corollary \ref{cor:compactness} provides conditions under which the sequence $\{x^k\}$ has cluster points.  One of these conditions is that $H$ is positive definite.  In such cases, the function $J_0$ is strongly convex and so the problem \eqref{s-dist} has a unique global solution $x^*$, meaning that the entire sequence converges to $x^*$.  We formalize this conclusion with the following theorem.

\begin{theorem}\label{thm:H pd implies cvg}
  Suppose that $H$ is positive definite and the sequence $\{(x^k,\eps^k)\}$ is generated by IRWA with initial point $x^0 \in X$ and relaxation vector $\eps^0 \in \bR^l_{++}$.  Then, the problem \eqref{s-dist} has a unique global solution $x^*$ and $x^k\rightarrow x^*$.
\end{theorem}

\begin{proof}
  Since $H$ is positive definite, the function $J(x,\eps)$ is strongly convex in $x$ for all $\eps\in\R^l_+$. In particular, $J_0$ is strongly convex and so \eqref{s-dist} has a unique global solution $x^*$.  By Corollary \ref{cor:compactness}, the set $L(J(x^0,\eps^0),\eps^0)$ is compact, and, by  Lemma~\ref{Gfun}, the sequence $J(x^k,\eps^k)$ is decreasing; hence, $\{x^k\}\subset L(J(x^0,\eps^0),\eps^0)$.  Therefore, the set $\{\dist{Ax^k+b}{C}\}_{k\in S}$ is bounded and $\ker(H)\cap\ker(A)\subset\ker(H)=\{0\}$, and so, by Theorem \ref{main}, the subsequence $\{x^k\}_{k\in S}$ has a cluster point $\bx$ satisfying $0\in\sd J_0(\bx)+\ncone{\bx}{X}$.  But the only such point is $\bx=x^*$, and hence $x^k\overset{S}{\rightarrow}x^*$. 

  Since the sequence $\{J(x^k,\eps^k)\}$ is monotonically decreasing and bounded below by Corollary \ref{cor:compactness}, it has a limit $\tJ$.  Since $x^k\overset{S}{\rightarrow}x^*$, we have $\tJ=\min_{x\in X}J_0(x)$.  Let $\tS$ be any subsequence of $\bN$. Since $\{x^k\}_{k\in\tS}\subset L(J(x^0,\eps^0),\eps^0)$ (which is compact by Corollary \ref{cor:compactness}(2)), this subsequence has a further subsequence $\tS_0\subset \tS$ such that $x^k\overset{\tS_0}{\rightarrow}\bx$ for some $\bx\in X$.  For this subsequence, $J(x^k,\eps^0)\overset{\tS_0}{\rightarrow}\tJ$, and, by continuity, $J(x^k,\eps^0)\overset{\tS_0}{\rightarrow} J(\bx,0)=J_0(\bx)$. Hence, $\bx=x^*$ by uniqueness. Therefore, since every subsequence of $\{x^k\}$ has a further subsequence that converges to $x^*$, it must be the case that the entire sequence converges to $x^*$.
\end{proof}

\subsection{Complexity of IRWA}

A point $\tx\in X$ is an $\veps$-optimal solution to \eqref{s-dist} if
\beq\label{eps opt sol}
  J_0(\tx)\le \inf_{x\in X} J_0(x)+\veps.
\eeq
In this section, we prove the following result.

\begin{theorem}\label{thm: IRWA complexity}
  Consider the problem \eqref{s-dist} with $X=\Rn$ and $H$ positive definite.  Let $\veps>0$ and $\eps\in\R^l_{++}$ be such that 
  \beq\label{eps choice}
    \onorm{\eps}\le\veps/2
    \quad\mbox{and}\quad
    \veps\le 4l\tveps,
  \eeq
  where $\tveps:=\min_{i\in\cI} \eps_i$.  Suppose that the sequence $\{(x^k,\eps^k)\}$ is generated by IRWA with initial point $x^0 \in \Rn$ and relaxation vector $\eps^0 = \eps \in \bR^l_{++}$, and that the relaxation vector is kept fixed so that $\eps^k=\eps$ for all $k\in\bN$.  Then, in at most $O(1/\veps^2)$ iterations, $x^k$ is an $\veps$-optimal solution to \eqref{s-dist}, i.e., \eqref{eps opt sol} holds with $\tx=x^k$.
\end{theorem}

The proof of this result requires a few preliminary lemmas. For ease of presentation, we assume that the hypotheses of Theorem \ref{thm: IRWA complexity} hold throughout this section.  Thus, in particular, Corollary \ref{cor:compactness} and the strict convexity and coercivity of $J$ tells us that there exists $\tau>0$ such that
\beq\label{tau}
  \tnorm{x^k-x^\eps} \le \tau\quad\mbox{for all}\ k\in\bN,
\eeq
where $x^\eps$ is the solution to $\min_{x\in\Rn} J(x,\eps)$.  Let $w_i$ for $i\in\cI$ and $\tA$ be given as in \eqref{w} and \eqref{tA}, respectively.  In addition, define
\[
  \begin{aligned}
    R_i(r_i)&:=\frac{r_i}{\sqrt{\tnorm{r_i}^2+\eps_i^2}},\quad r_i(x):=(I-P_{C_i})(A_ix+b_i) \mbox{ for } i \in \cI\\ 
    \mbox{and}\quad      
    u(x,\eps)&:=\nabla\varphi(x)+\sum_{i\in\cI}w_i(x,\eps)A_i^Tr_i(x).
  \end{aligned}
\]
Recall that
\[
  \begin{aligned}
     &\ \partial_x J(x,\epsilon) \\
    =&\ \nabla\varphi(x) + \sum_{i\notin \cA(Ax+b)}w_i(x,\epsilon) A^T_i r_i(x) + \sum_{i\in \cA(Ax+b)} w_i(x,\epsilon) A^T_i  ( \bB_2 \cap N(A_ix+b_i|C_i)),
  \end{aligned}
\]
so that $u(x,\epsilon) \in \partial_x J (x,\epsilon)$.  It is straightforward to show that, for each $i\in\cI$, we have
\[
  \nabla_{r_i}R_i(r_i)=\frac{1}{\sqrt{\tnorm{r_i}^2+\eps_i^2}} \left(I-\frac{r_ir_i^T}{\tnorm{r_i}^2+\eps_i^2}\right)
\]
so that 
\beq\label{R lip}
  \tnorm{\nabla_{r_i}R_i(r_i)}\le1/\eps_i\quad \forall\, r_i.
\eeq
Consequently, for each $i \in \cI$, the function $R_i$ is globally Lipschitz continuous with Lipschitz constant
$1/\eps_i$. This allows us to establish a similar result for the mapping $u(x,\eps)$
as a function of $x$, which we prove as our next result. For convenience, we use
\[
\bu:=u(\bx,\eps),\ \hu:=u(\hx,\eps),\mbox{ and }u^k:=u(x^k,\eps),
\]
and similar shorthand for $w_i(x,\eps_i)$, $W(x,\eps)$, and $r_i(x)$.

\begin{lemma}\label{u lip}
  Let the hypotheses of Theorem \ref{thm: IRWA complexity} hold.  Moreover, let $\lam$ be the largest eigenvalue of $H$ and $\sig_1$ be an upper bound on all singular values of the matrices $A_i$ for $i\in\cI$.  Then, as a function of $x$, the mapping $u(x,\eps)$ is globally Lipschitz continuous with Lipschitz constant $\beta:=\lam+l\sig_1^2/\tveps$.
\end{lemma}
\begin{proof}
  By Theorem \ref{projection theorem}, for all $\{\bx,\hx\}\subset\Rn$, we have
  \beq\label{r lip}
    \tnorm{\br_i-\hr_i}\le\tnorm{A_i(\bx-\hx)}\le \sig_1\tnorm{\bx-\hx}.
  \eeq
  Therefore,
  \begin{eqnarray*}
    \tnorm{\bu-\hu}
      &=  & \tnorm{H(\bx-\hx)+\sum_{i\in\cI}A_i^T(R_i(\br_i)-R_i(\hr_i))} \\ 
      &\le& \tnorm{H}\tnorm{\bx-\hx}+\tfrac{1}{\teps}\sum_{i\in\cI}\tnorm{A_i}\tnorm{\br_i-\hr_i} \\
      &\le& \tnorm{H}\tnorm{\bx-\hx}+\tfrac{1}{\teps}\sum_{i\in\cI}\tnorm{A_i}^2\tnorm{\bx-\hx} \\ 
      &\le& (\lam+l\sig_1^2/\tveps)\tnorm{\bx-\hx},
  \end{eqnarray*}
  where the first inequality follows from \eqref{R lip}, the second from \eqref{r lip}, and the last from the fact that the $2$-norm of a matrix equals its largest singular value.
\end{proof}

By Lemma \ref{u lip} and the subgradient inequality, we obtain the bound
\beq\label{q bd for J}
  0\le J(\bx,\eps)-J(\hx,\eps)-\ip{\hu}{\bx-\hx} \le \ip{\bu-\hu}{\bx-\hx}\le \beta\tnorm{\bx-\hx}^2.
\eeq
Moreover, by Part (5) of Lemma \ref{irwa properties}, we have
\[
  -\tA^T W_k q^k = - \tA^T W_k \tA( x^{k+1} - x^k) = u^k \in \partial_x J(x^k, \eps).
\]
If we now define $D_k:=\tA^T W_k \tA$, then $x^k-x^{k+1}=D_k^{-1}u^k$ and
\beq\label{qux}
  (q^k)^TW_kq^k = (x^k-x^{k+1})^TD_k(x^k-x^{k+1}) = (u^k)^TD_k^{-1}u^k.
\eeq
This gives the following bound on the decrease in $J$ when going from $x^k$ to $x^{k+1}$.

\begin{lemma}\label{decrease in J}
  Let the hypotheses of Lemma \ref{u lip} hold. Then,
  \[
    J(x^{k+1}, \eps) - J(x^k, \eps) \le - \alpha \|u^k\|_2^2,
  \]
  where $\alf:=\tveps/(2\sig_0^2)$ with $\sig_0$ the largest singular value of $\tA$.
\end{lemma}
\begin{proof}
  By Lemma~\ref{Gfun} and \eqref{qux}, we have
  \[
    J(x^{k+1}, \eps) - J(x^k, \eps) \le - \thalf (q^k)^TW_k q^k = -\thalf (u^k)^T D_k^{-1} u^k.
  \]
  Since the $\tnorm{D_k}\le \|W^{1/2}_k\|_2^2\|\tA\|_2^2$, we have that the largest eigenvalue of $D_k$ is bounded above by $\sigma_0^2/\tveps$.  This implies $\thalf(u^k)^T D_k^{-1}u^k \ge \alf\tnorm{u^k}^2$, which gives the result.
\end{proof}

The following theorem is the main tool for proving Theorem \ref{thm: IRWA complexity}.

\begin{theorem}\label{main complexity theorem}
  Let the hypotheses of Lemma \ref{decrease in J} hold, and, as in \eqref{tau}, let $x^\eps$ be the solution to $\min_{x\in\Rn} J(x,\eps)$.  Then,
  \beq\label{basic complexity bound}
    J(x^k, \epsilon)-J(x^\epsilon, \epsilon) \le \frac{32l^2\sig_0^2\tau^2}{k\veps}\left[\frac{\tnorm{u^\eps}\veps+\tau(\lam\veps +l\sig_1^2)}{\tnorm{u^\eps}\veps+\tau(\lam\veps +4l^2\sig_1^2)+8l\tau\sig_0^2/k}\right].
  \eeq
  Therefore, IRWA requires $O(1/\veps^2)$ iterations to reach $\veps$-optimality for $J(x,\eps)$, i.e., 
  \[
    J(x^k,\epsilon) - J(x^\eps,\eps) \le \veps.
  \]
\end{theorem}
\begin{proof}
  Set $\delta^j := J(x^j, \epsilon)-J(x^\epsilon, \epsilon)$ for all $j\in\bN$.  Then, by Lemma \ref{decrease in J},
  \beq\label{eq.rate.1} 
    \begin{aligned}
      0 \le \delta^{j+1} 
      &=  J(x^{j+1}, \epsilon) - J(x^\epsilon, \epsilon) \\
      &\le J(x^j, \epsilon)-J(x^\epsilon, \epsilon) - \alpha \| u^j\|_2^2 = \delta^j - \alpha \| u^j\|_2^2 \le\del^j.
    \end{aligned}
  \eeq
   If for some $j < k$ we have $\delta^j = 0$, then \eqref{eq.rate.1} implies that $\delta^k = 0$ and $u^k=0$, which in turn implies that $x^{k+1}=x^\eps$ and the bound \eqref{basic complexity bound} holds trivially.  In the remainder of the proof, we only consider the nontrivial case where $\delta^j > 0$ for $j=0,...,k-1$. 

  Consider $j \in \{0,\dots,k-1\}$.  By the convexity of $J$ and \eqref{tau}, we have
  \[
    \delta^j = J(x^j, \epsilon) - J(x^\epsilon, \epsilon) \le (u^j)^T(x^j - x^\epsilon) \le \|u^j\|_2\|x^j-x^\epsilon\| \le \tau \|u^j\|_2.
  \]
  Combining this with \eqref{eq.rate.1}, gives  
  \[
    \delta^{j+1} \le \delta^j - \frac{\alpha}{\tau^2} (\delta^j)^2.
  \]
  Dividing both sides by $\delta^{j+1} \delta^j$ and noting that $\frac{\delta^j}{\delta^{j+1}} \ge 1$ yields 
  \beq\label{eq.rate.2}  
    \frac{1}{\delta^{j+1}} - \frac{1}{\delta^j} \ge \frac{\alpha}{\tau^2} \frac{\delta^j}{\delta^{j+1}} \ge \frac{\alpha}{\tau^2}.
  \eeq
  Summing both sides of \eqref{eq.rate.2} from 0 to $k-1$, we obtain
  \beq\label{eq.rate.3} 
    \frac{1}{\delta^k} \ge \frac{\alpha k}{\tau^2} + \frac{1}{\delta^0} = \frac{\alpha\delta^0 k+\tau^2}{\delta^0\tau^2},
  \eeq
  or, equivalently,
  \beq\label{eq.rate.4} 
    \delta^k\le  \frac{\delta^0\tau^2}{\alpha\delta^0 k+\tau^2}.
  \eeq
  The inequality \eqref{q bd for J} implies that
  \[
    \delta^0 = J(x^0, \epsilon)-J(x^\epsilon, \epsilon) \le (u^\epsilon)^T(x^0-x^\epsilon) + \beta \| x^0- x^\epsilon\|_2^2 \le \tau(\|u^\epsilon\|_2+\beta \tau),
  \]
  which, together with \eqref{eq.rate.3}, implies that 
  \[ 
    \frac{\alpha\delta^0k+\tau^2}{\delta^0\tau^2} \ge \frac{\alpha k}{\tau^2} + \frac{1}{\tau(\|u^\epsilon\|_2+\beta \tau)}.
  \]
  Rearranging, one has 
  \[  
    \frac{\tau^2\del^0}{\alpha k\del^0 +\tau^2} \le \frac{\tau^2(\|u^\epsilon\|_2+\beta \tau )}{ \alpha k(\|u^\epsilon\|_2+\beta \tau) + \tau }.
  \]
  Substituting in $\beta= \lam+l\sig_1^2/\tveps$ and $\alf=\tveps/(2\sig_0^2)$ defined in Lemmas~\ref{u lip} and \ref{decrease in J}, respectively, and then combining with \eqref{eq.rate.4} gives
  \[
    \delta^k\le\frac{\tau^2(\tnorm{u^\eps}+\tau(\lam +l\sig_1^2/\tveps))}{(\tveps/(2\sig_0^2))k(\tnorm{u^\eps}+\tau(\lam +l\sig_1^2/\tveps))+\tau} = \frac{2\sig_0^2\tau^2}{k\tveps}\left[\frac{\tnorm{u^\eps}\tveps+\tau(\lam\tveps +l\sig_1^2)}{\tnorm{u^\eps}\tveps+\tau(\lam\tveps +l\sig_1^2)+2\tau\sig_0^2/k}\right].
  \]
  Finally, using the inequalities $\veps\le 4l\tveps$ and $\tveps\le\veps$ (recall \eqref{eps choice}) gives
  \[
    \delta^k \le \frac{32l^2\sig_0^2\tau^2}{k\veps}\left[\frac{\tnorm{u^\eps}\veps+\tau(\lam\veps +l\sig_1^2)}{\tnorm{u^\eps}\veps+\tau(\lam\veps +4l^2\sig_1^2)+8l\tau\sig_0^2/k}\right],
  \]
  which is the desired inequality.
\end{proof}

We can now prove Theorem \ref{thm: IRWA complexity}.
\smallskip

\begin{proof}[Theorem~\ref{thm: IRWA complexity}]
  Let $x^* = \arg\min_{x\in\Rn}J_0(x)$.  Then, by convexity in $\eps$,
  \begin{equation*}
    \begin{aligned}
      J(x^\eps,\eps) - J(x^*,0) 
        &\le [\partial_\eps J(x^\eps,\eps)]^T (\eps-0) \\ 
        &=   \sum_{i\in\cI} \frac{\eps_i^2}{ \sqrt{ \| r_i(x^\eps)\|_2^2 + \eps_i^2} } \le \sum_{i\in\cI} \eps_i =\onorm{\eps}\le \veps/2.
    \end{aligned}
  \end{equation*}
  By Theorem \ref{main complexity theorem}, IRWA needs $O(1/\veps^2)$ iterations to reach 
  \[
    J(x^k,\eps)-J(x^\eps,\eps) \le \veps/2.
  \]
  Combining these two inequalities yields the result.
\end{proof}

\section{An Alternating Direction Augmented Lagrangian Algorithm}\label{sec.adal}

For comparison with IRWA, we now describe an alternating direction augmented Lagrangian method for solving problem \eqref{s-dist}.  This approach, like IRWA, can be solved by matrix-free methods.  Defining
\[
  \hat{J}(x,p) := \varphi(x)+\dist{p}{C},
\]
where $\dist{p}{C}$ is defined as in \eqref{prod dist}, the problem \eqref{s-dist} has the equivalent form
\begin{equation}\label{Jhat}
  \min_{x\in X,p}\ \hat{J}(x,p)\ \mbox{subject to}\ Ax+b=p,
\end{equation}
where $p:=(p_1^T,\dots,p_l^T)^T$.  In particular, note that $J_0(x)=\hJ(x,Ax+b)$.  Defining dual variables $(u_1,\dots,u_l)$, a partial Lagrangian for \eqref{Jhat} is given by
\[
  L(x,p,u) :=  \hat{J}(x,p) + \ip{u}{Ax+b-p}+\indicator{x}{X},
\]
and the corresponding augmented Lagrangian, with penalty parameter $\mu > 0$, is
\[
  L(x,p,u,\mu) := \hat{J}(x,p) + \tfrac{1}{2\mu}\tnorm{Ax + b- p + \mu u}^2 - \tfrac{\mu}{2}\|u\|^2_2+\indicator{x}{X}.
\]
(Observe that due to their differing numbers of inputs, the Lagrangian value $L(x,p,u)$ and augmented Lagrangian value $L(x,p,u,\mu)$ should not be confused with each other, nor with the level set value $L(\alpha,\epsilon)$ defined in Theorem~\ref{thm:boundedness}.)

We now state the algorithm.
\medskip

\underline{\bf{Alternating Direction Augmented Lagrangian Algorithm (ADAL)}}
\begin{enumerate}
  \item[Step 0:] (Initialization)
  Choose an initial point $x^0\in X$, dual vectors $u_i^0 \in \Rmi$ for $i \in \cI$, and penalty parameter $\mu > 0$.  Let $\sigma \geq 0$ and $\sigma'' \geq 0$ be two scalars which serve as termination tolerances for the stepsize and constraint residual, respectively.  Set $k:=0$.
  \item[Step 1:] (Solve the augmented Lagrangian subproblems for $(x^{k+1},p^{k+1})$)\\
  \noindent
  Compute a solution $p^{k+1}$ to the problem
  \[
    \cL_p{(x^k,p,u^k,\mu)}:\quad \min_{p} L(x^k,p,u^k,\mu),
  \]
  and a solution $x^{k+1}$ to the problem
  \[
    \cL_x{(x,p^{k+1},u^k,\mu)}:\quad \min_{x } L(x,p^{k+1},u^k,\mu).
  \]
  \item[Step 2:] (Set the new multipliers $u^{k+1}$)\\
  \noindent
  Set 
  \[
    u^{k+1} := u^k + \tfrac{1}{\mu}(Ax^{k+1} + b - p^{k+1}).
  \]
  \item[Step 3:] (Check stopping criteria)\\
  \noindent
  If $\norm{x^{k+1} - x^k}_2 \leq \sigma$ and $\dnorm{Ax^{k+1} + b - p^{k+1}} \leq \sigma''$, then stop; else, set $k := k+1$ and go to Step 1.
\end{enumerate}

\begin{rem}
  As for IRWA, one can also base the stopping criteria of Step~3 on a percent reduction in duality gap; recall Remark~\ref{rem duality gap}.
\end{rem}

\subsection{Properties of $\cL_p{(x,p,u,\mu)}$ and $\cL_x{(x,p,u,\mu)}$}

Before addressing the convergence properties of the ADAL algorithm, we discuss properties of the solutions to the subproblems $\cL_p{(x,p,u,\mu)}$ and $\cL_x{(x,p,u,\mu)}$.

The subproblem $\cL_p{(x^k,p,u^k,\mu)}$ is separable.  Defining
\begin{equation*}
  s^k_i := A_ix^k + b_i + \mu u^k_i\quad\forall\, i \in \cI,
\end{equation*}
the solution of $\cL_p{(x^k,p,u^k,\mu)}$ can be written explicitly, for each $i \in \cI$, as
\begin{equation}\label{defn pk}
  p^{k+1}_i := \begin{cases} P_{C_i}(s^k_i) & \text{if $\tdist{s^k_i}{C_i} \leq \mu$} \\ s^k_i - \frac{\mu}{\tdist{s^k_i}{C_i}}(s^k_i - P_{C_i}(s^k_i)) & \text{if $\tdist{s^k_i}{C_i} > \mu$.} \end{cases}
\end{equation}
Subproblem $\cL_x{(x,p^{k+1},u^k,\mu)}$, on the other hand, involves the minimization of a convex quadratic over $X$, which can be solved by matrix-free methods.

Along with the dual variable estimates $\{u_i^k\}$, we define the auxiliary estimates
\begin{equation*}
 \hu^{k+1} := u^{k+1} - \tfrac{1}{\mu}q^k,
 \quad\mbox{where}\quad
 q^k := A(x^{k+1}  - x^k)
 \quad\mbox{as in IRWA Step 2.}
\end{equation*}
First-order optimality conditions for \eqref{Jhat} are then given by
\begin{subequations}\label{eq.kkt}
  \begin{align}
    0 & \in \partial \dist{p}{C} - u, \label{eq.dualfeas2}\\
    0 &\in  \nabla \varphi(x) + A^Tu + N(x|X), \label{eq.dualfeas1} \\
    0 &= Ax + b - p, \label{eq.primfeas} 
  \end{align}
\end{subequations}
or, equivalently,
\[
  0\in\sd J_0(x)=\nabla\varphi(x)+A^T\sd \dist{\cdot}{C}(Ax+b)+\ncone{x}{X}.
\]
The next lemma relates the iterates and these optimality conditions.

\begin{lemma}\label{lem:adal 1}
  Suppose that the sequence $\{(x^k,p^k,u^k)\}$ is generated by ADAL with initial point $x^0 \in X$.  Then, for all $k \in \bN$, we have
  \begin{equation}\label{adal 1}
    \hu^{k+1}\in\sd \dist{p^{k+1}}{C}
    \ \mbox{and}\ 
    -\tfrac{1}{\mu}A^Tq^k \in \nabla \varphi(x^{k+1})  +A^T\hu^{k+1}+\ncone{x^{k+1}}{X}.
  \end{equation}
  Therefore,
  \begin{equation*}\label{adal 2}
    -\tfrac{1}{\mu}A^Tq^k \in \nabla \varphi(x^{k+1})  +A^T\sd \dist{p^{k+1}}{C}+\ncone{x^{k+1}}{X}.
  \end{equation*}
  Moreover, for all $k\geq1$, we have
  \begin{equation}\label{adal 3}
    \dnorm{\hu^{k}}\le 1,\quad \dnorm{s^{k}}\le\mu,\ \mbox{ and }\ \dnorm{p^{k}}\le\hmu,
  \end{equation}
  where $\hmu:=\max\{\mu,\sup_{\dnorm{s}\le\mu}\dnorm{P_C(s)}\}<\infty$.
\end{lemma}
\begin{proof}
  By ADAL Step 1, the auxiliary variable $p^{k+1}$ satisfies
  \[
    0 \in \partial \dist{p^{k+1}}{C} - u^k - \tfrac{1}{\mu} (Ax^k + b - p^{k+1}),
  \]
  which, along with ADAL Step 2, implies that
  \[
    \begin{aligned}
      u^{k+1} 
      &\in \partial \dist{p^{k+1}}{C} + \tfrac{1}{\mu}(Ax^{k+1}  + b - p^{k+1} ) - \frac{1}{\mu} (Ax^k + b - p^{k+1} ) \\
      &=   \partial \dist{p^{k+1}}{C} + \tfrac{1}{\mu} q^k.
    \end{aligned}
  \]
  Hence, the first part of \eqref{adal 1} holds.  Then, again by ADAL Step 1, $x^{k+1}$ satisfies
  \begin{equation*}\label{eq.x}
    0 \in \nabla \varphi(x^{k+1}) + \frac{1}{\mu}A^T( {A} x^{k+1}   + b - p^{k+1}  + \mu u^k ) + N(x^{k+1} |X).
  \end{equation*}
  which, along with ADAL Step 2, implies that
  \begin{equation}\label{opt.x}
    0 \in \nabla \varphi(x^{k+1}) + A^T u^{k+1} + N(x^{k+1}|X).
  \end{equation}
  Hence, the second part of \eqref{adal 1} holds.

  The first bound in \eqref{adal 3} follows from the first part  of \eqref{adal 1}.  The second bound in \eqref{adal 3} follows from the first bound and the fact that for $k\in\bN$ we have
  \[
    s^k=Ax^k+b+\mu u^k=\mu u^{k+1}-q^k=\mu \hu^{k+1}.
  \]
  As for the third bound, note that if, for some $i\in\cI$, we have $\tdist{s^{k-1}_i}{C_i} \leq \mu$, then, by \eqref{defn pk}, we have $\tnorm{p^{k}_i}\le\hmu$; on the other hand, if $\tdist{s^{k-1}_i}{C_i}> \mu$ so that $0<\xi:=\mu/\tdist{s^{k-1}_i}{C_i}<1$, then, by \eqref{defn pk} and the second bound in \eqref{adal 3},
  \[
    \tnorm{p^{k}_i}\le (1-\xi)\tnorm{s^{k-1}_i}+\xi \hmu\le\hmu.
  \]
  Consequently, $\dnorm{p^{k}}=\sup_{i\in\cI}\tnorm{p^{k}_i}\le\hmu$.
\end{proof}

For the remainder of our discussion of ADAL, we define the residuals
\begin{equation*}
  z^{k+1} := Ax^{k+1} + b - p^{k+1}.
\end{equation*}
Lemma \ref{lem:adal 1} tells us that the deviation of $(p^{k+1},\hu^{k+1})$ from satisfying the first-order optimality conditions for \eqref{eq.kkt} can be measured by
\begin{equation}
  E^{k+1} =\max \{\norm{q^k}, \dnorm{z^{k+1}}\}.
\end{equation}

\subsection{Convergence of ADAL}

In this section, we establish the global convergence properties of the ADAL algorithm.  The proofs in this section are standard for algorithms of this type (e.g., see \cite{BPCPE11}), but we include them for the sake of completeness.  We make use of the following standard assumption.

\begin{assu}\label{adal assumption}
  There exists a point $(x^*,p^*,u^*)$ satisfying \eqref{eq.kkt}.
\end{assu}

Since \eqref{Jhat} is convex, this assumption is equivalent to the existence of a minimizer.  Notice that $(x^*, p^*)$ is a minimizer of the convex function $L(x,p,u^*)$ over $X$.  We begin our analysis by providing useful bounds on the optimal primal objective value.

\begin{lemma}
  Suppose that the sequence $\{(x^k,p^k,u^k)\}$ is generated by ADAL with initial point $x^0 \in X$.  Then, under Assumption \ref{adal assumption}, we have for all $k\in\bN$ that
  \begin{equation}\label{eq.upper}
   (u^*)^T z^{k+1}\ge   \hat J(x^*,p^*) - \hat J(x^{k+1} ,p^{k+1}) \geq (u^{k+1} )^T z^{k+1}  - \frac{1}{\mu} (q^k)^T (p^* - p^{k+1} ).
  \end{equation}
\end{lemma}
\begin{proof}
  Since $(x^*,p^*,u^*)$ is a saddle point of $L$, it follows that $Ax^* + b - p^* = 0$, which implies by the fact that $x^{k+1}\in X$ that
  \[
    \hat J(x^*,p^*) = L(x^*,p^*,u^*) \leq L(x^{k+1} ,p^{k+1} ,u^*).
  \]
  Rearranging, we obtain the first inequality in \eqref{eq.upper}.

  We now show the second inequality in \eqref{eq.upper}.  Recall that Steps 1 and 2 of ADAL tell us that \eqref{opt.x} holds for all $k\in\bN$.  Therefore, $x^{k+1} $ is first-order optimal for
  \begin{equation*}
    \min_{x\in X }\ \varphi(x) + (u^{k+1})^TAx.
  \end{equation*}
  Since this is a convex problem and $x^*\in X$, we have 
  \begin{equation}\label{eq.weird}
    \varphi(x^*)+(u^{k+1})^TAx^*\ge \varphi(x^{k+1}) + (u^{k+1})^TAx^{k+1} .
  \end{equation}
  Similarly, by the first expression in \eqref{adal 1}, $p^{k+1} $ is first-order optimal for
  \begin{equation*}
    \min_{p }\ \dist{p}{C} - (\hu^{k+1})^Tp.
  \end{equation*}
  Hence, by the convexity of this problem, we have
  \begin{equation}\label{eq.weird2}
    \dist{p^*}{C} - (\hu^{k+1})^Tp^*\ge \dist{p^{k+1}}{C} - (\hu^{k+1})^Tp^{k+1}.
  \end{equation}
  By adding \eqref{eq.weird} and \eqref{eq.weird2}, we obtain
  \begin{equation*} 
    \begin{aligned}
         &\ \hat J(x^*,p^*) - \hat J(x^{k+1} , p^{k+1} ) \\
      \ge&\ (\hu^{k+1})^T(p^*-p^{k+1} )+(u^{k+1})^TA(x^{k+1} -x^*) \\
        =&\ (u^{k+1})^T(p^*-p^{k+1} )-\tfrac{1}{\mu}(q^k  )^T (p^* - p^{k+1} )+(u^{k+1})^TA(x^{k+1} -x^*) \\
        =&\ (u^{k+1})^T\left((p^*-Ax^{*})-b)-(p^{k+1}-Ax^{k+1}-b) \right)-\tfrac{1}{\mu}(q^k  )^T (p^* - p^{k+1}) \\
        =&\ (u^{k+1} )^T z^{k+1} - \tfrac{1}{\mu} (q^k)^T (p^* - p^{k+1}),
    \end{aligned}
  \end{equation*}
  which completes the proof.
\end{proof}
 
Consider the distance measure to $(x^*,u^*)$ defined by
\begin{equation*}
  \om^k := \tfrac{1}{\mu}\tnorm{A(x^k-x^*)}^2 + \mu\tnorm{u^k-u^*}^2.
\end{equation*}
In our next lemma, we show that this measure decreases monotonically. 

\begin{lemma}\label{lem.ubounded}
  Suppose that the sequence $\{(x^k,p^k,u^k)\}$ is generated by ADAL with initial point $x^0 \in X$.  Then, under Assumption \ref{adal assumption} holds, we have for all $k \geq 1$ that 
  \begin{equation}\label{eq.phibd}
    \tfrac{1}{\mu} (\tnorm{z^{k+1} }^2 + \tnorm{q^k}^2) + 2(x^{k+1} -x^k)^T H(x^{k+1} -x^k)  \leq  \om^k - \om^{k+1}.
  \end{equation}
\end{lemma}
\begin{proof}
  By using the extremes of the inequality \eqref{eq.upper} and rearranging, we obtain
  \begin{equation*}
    (u^{k+1} -u^*)^T z^{k+1} - \tfrac{1}{\mu}(q^k)^T(p^* - p^{k+1}) \leq 0 .
  \end{equation*}
  Since $(x^*,p^*,u^*)$ is a saddle point of $L$, and so $Ax^* + b= p^*$, this implies
  \begin{equation}\label{eq.bd}
    (u^{k+1} - u^*)^T z^{k+1} - \tfrac{1}{\mu}(q^k  )^T z^{k+1}  + 
    \tfrac{1}{\mu}(x^{k+1} -x^k)^T A^T A(x^{k+1}  - x^*)  \leq 0.
    \end{equation}
  The update in Step 2 yields $u^{k+1} = u^k+ \frac{1}{\mu}z^{k+1} $, so we have
  \begin{equation}\label{eq.expand}
    (u^{k+1}  - u^*)^T z^{k+1} = \left[(u^k_i - u^*)^T z^{k+1} + \tfrac{1}{2\mu}\tnorm{z^{k+1}}^2\right] + \tfrac{1}{2\mu}\tnorm{z^{k+1}}^2.
  \end{equation}
  Let us now consider the first grouped term in \eqref{eq.expand}.  From ADAL Step 2, we have $z^{k+1} = \mu(u^{k+1} -u^k)$, which gives
  \begin{align}
    (u^k - u^*)^T z^{k+1}  + \tfrac{1}{2\mu}\tnorm{z^{k+1}}^2
      = &\ \mu(u^k - u^*)^T (u^{k+1} -u^k) + \tfrac{\mu}{2}\|u^{k+1} -u^k\|^2_2 \nonumber \\
      = &\ \mu(u^k - u^*)^T (u^{k+1}  - u^*) - \mu(u^k - u^*)^T (u^{k} - u^*)\nonumber \\
        &\ + \tfrac{\mu}{2}\|(u^{k+1}  - u^*) - (u^k - u^*)\|^2_2 \nonumber \\
      = &\ \tfrac{\mu}{2}(\|u^{k+1}  - u^*\|^2_2 - \|u^k - u^*\|^2_2). \label{eq.1}
  \end{align}
  Adding the final term $\frac{1}{2\mu}\tnorm{z^{k+1}}^2$ in \eqref{eq.expand} to the second and third terms in \eqref{eq.bd},
  \begin{align}
    &\ \tfrac{1}{\mu} \left(\tfrac{1}{2}\tnorm{z^{k+1}}^2 - (q^k  )^T z^{k+1}  
    + (x^{k+1} -x^k)^T A^T A(x^{k+1} -x^*)\right)  \nonumber \\
    =&\ \tfrac{1}{\mu}\left(\tfrac{1}{2}\tnorm{z^{k+1}}^2 - (q^k  )^T z^{k+1}  
    + (x^{k+1} -x^k)^T A^T A((x^{k+1} -x^k) + (x^k - x^*))\right)  \nonumber \\
    =&\ \tfrac{1}{\mu}\left(\tfrac{1}{2}\tnorm{z^{k+1}  - q^k  }^2 + \tfrac{1}{2} \|q^k  \|^2_2 
    + (x^{k+1} -x^k)^T A^T A(x^k - x^*)\right)  \nonumber \\
    =&\ \tfrac{1}{\mu}(\tfrac{1}{2}\tnorm{z^{k+1}  - q^k  }^2 + 
    \tfrac{1}{2} \|A((x^{k+1} -x^*) - (x^k - x^*))\|^2_2 \nonumber\\
       &\ + ((x^{k+1} -x^*) - (x^k - x^*))^T A^T A(x^k - x^*))  \nonumber \\
    =&\ \tfrac{1}{2\mu} \left(\tnorm{z^{k+1}  - q^k  }^2 + 
    \|A(x^{k+1} -x^*)\|^2_2 - \|A(x^k-x^*)\|^2_2\right) \label{eq.2}
  \end{align}
  From \eqref{eq.expand}, \eqref{eq.1}, and \eqref{eq.2}, we have that \eqref{eq.bd} reduces to
  \begin{equation*}
    \om^{k+1}  - \om^k \leq - \frac{1}{\mu} \tnorm{z^{k+1}  - q^k  }^2.
  \end{equation*}
  Since \eqref{opt.x} holds for $k \geq 1$, we have
  \begin{equation*} 
    -(v^{k+1}-v^k) = H(x^{k+1} -x^k) +A^T (u^{k+1}  - u^k),
  \end{equation*} 
  for some $v^{k+1}\in N(x^{k+1}|X)$ and $v^k\in N(x^k|X)$.
  Therefore,
  \begin{align} 
    (u^{k+1} -u^k)^T q^k 
    & = -(v^{k+1}-v^k)^T(x^{k+1}-x^k)-(x^{k+1}-x^k)^TH(x^{k+1}-x^k) \nonumber \\
    & \le -(x^{k+1}-x^k)^TH(x^{k+1}-x^k), \label{eq.herewego}
  \end{align}
  where the inequality follows since the normal cone operator $\ncone{\cdot}{C}$ is a monotone operator \cite{RTRW}.  Using this inequality in the expansion of the right-hand side of \eqref{eq.herewego} along with the equivalence $z^{k+1} = \mu(u^{k+1} -u^k)$, gives
  \begin{equation*} 
    \begin{aligned}
      \om^{k+1}  -\om^k
        &\leq -\tfrac{1}{\mu} \left(\|z^{k+1} \|^2_2 - 2\mu(u^{k+1} -u^k)^T q^k   + \|q^k  \|^2_2\right)  \\
        &\leq -\tfrac{1}{\mu} (\|z^{k+1} \|^2_2 + \|q^k  \|^2_2) 
        + 2(u^{k+1} _i-u^k_i)^T q_i^k   \\
        &\le   -\tfrac{1}{\mu} (\|z^{k+1} \|^2_2 + \|q^k  \|^2_2)-2(x^{k+1}-x^k)^TH(x^{k+1}-x^k),
    \end{aligned}
  \end{equation*} 
  as desired.
\end{proof}

We now state and prove our main convergence theorem for ADAL.

\begin{theorem}\label{thm:adal cvg}
  Suppose that the sequence $\{(x^k,p^k,u^k)\}$ is generated by ADAL with initial point $x^0 \in X$.  Then, under Assumption \ref{adal assumption}, we have
  \[
    \lim_{k\to\infty} q^k = 0 ,\quad \lim_{k\to\infty} z^{k+1}  = 0,\quad \mbox{and so}\quad \lim_{k\to\infty} E^k = 0.
  \]
  Moreover, the sequences $\{u^k\}$ and $\{Ax^k\}$ are bounded and 
  \begin{equation*} 
    \lim_{k\to\infty} \hat J(x^k, p^k) = \hat J(x^*, p^*) = J_0(x^*).
  \end{equation*} 
\end{theorem}
\begin{proof}
  Summing \eqref{eq.phibd} over all $k \geq 1$ yields
  \[
    \sum_{k=1}^\infty \left(2(x^{k+1} -x^k)^T H(x^{k+1} -x^k) + \tfrac{1}{\mu} (\|z^{k+1} \|^2_2 + \|q^k  \|^2_2)\right) \leq \om^1,
  \] 
  which, since $H\succeq0$, implies that $z^{k+1} \to 0$ and $q^k \to 0$.  Consequently, $E^k\to 0$.

  The sequence $\{u^k\}$ is bounded since $\hu^{k+1}+(1/\mu)q^k=u^{k+1}$, where $\{\hu^k\}$ is bounded by \eqref{adal 3} and $q^k\to 0$.  Similarly, the sequence $\{Ax^k\}$ is bounded since $\mu(u^{k+1}-u^k)+p^{k+1}-b=Ax^{k+1}$, where the sequence $\{p^k\}$ is bounded by \eqref{adal 3}.  Finally, by \eqref{eq.upper}, we have that $\hJ(x^k,p^k)\to \hJ(x^*,p^*)$ since both $z^k\to 0$ and $q^k\to 0$ while $\{p^k\}$ and $\{u^k\}$ are both bounded.
\end{proof}
  
\begin{cor}\label{cor:adal cluster}
  Suppose that the sequence $\{(x^k,p^k,u^k)\}$ is generated by ADAL with initial point $x^0 \in X$.  Then, under Assumption \ref{adal assumption}, every cluster point of the sequence $\{x^k\}$ is a solution to \eqref{s-dist}.
\end{cor}
\begin{proof}
  Let $\bx$ be a cluster point of $\{x^k\}$, and let $S\subset\bN$ be a subsequence such that $x^k\overset{S}{\to}\bx$. By \eqref{adal 3}, $\{p^k\}$ is bounded so we may assume with no loss in generality that there is a $\bp$ such that $p^k\overset{S}{\to}\bp$. Theorem \ref{thm:adal cvg} tells us that $A\bx+b=\bp$ and $\hJ(\bx,\bp)=J_0(x^*)$ so that $J_0(\bx)=\hJ(\bx,A\bx+b)=\hJ(\bx,\bp)=J_0(x^*)$.
\end{proof}

We now address the question of when the sequence $\{x^k\}$ has cluster points.  For the IRWA of the previous section this question was answered by appealing to Theorem \ref{thm:boundedness} which provided necessary and sufficient conditions for the compactness of the lower level sets of the function $J(x,\eps)$.  This approach also applies to the ADAL algorithm, but it is heavy handed in conjunction with Assumption \ref{adal assumption}. In the next result we consider two alternative approaches to this issue.

\begin{prop}
  Suppose that the sequence $\{(x^k,p^k,u^k)\}$ is generated by ADAL with initial point $x^0 \in X$.  If either
  \begin{enumerate}
    \item[(a)] $\left[\bx\in X^\infty\cap\ker{(H)}\cap A^{-1}C^\infty\mbox{ satisfies }g^T\bx\le 0\right]\iff \bx=0$, or
    \item[(b)] Assumption \ref{adal assumption} holds and
    \begin{equation}\label{weak boundedness}
      \left[\tx\in X^\infty\cap\ker(H)\cap\ker(A)\mbox{ satisfies }g^T\tx\le 0\right]\iff \tx=0,
    \end{equation}
  \end{enumerate}
  then $\{x^k\}$ is bounded and every cluster point of this sequence is a solution to \eqref{s-dist}. 
\end{prop}

\begin{proof}
  Let us first assume that (a) holds.  By Theorem \ref{thm:boundedness}, the condition in (a) (recall \eqref{boundedness}) implies that the set $L(J(x^0,0),0)$ is compact.  Hence, a solution $x^*$ to \eqref{s-dist} exists.  By \cite[Theorem 23.7]{RTR}, there exist $p^*$ and $u^*$ such that $(x^*,p^*,u^*)$ satisfies \eqref{eq.kkt}, i.e., Assumption \ref{adal assumption} holds.  Since
  \begin{align*}
    J(x^{k},0)
      &=   \varphi(x^k)+\dist{Ax^k+b}{C} \\ 
      &=   \varphi(x^k)+\norm{(Ax^k+b)-P_C(Ax^k+b)} \\
      &\le \varphi(x^k)+\norm{(Ax^k+b)-p^k}+\norm{p^k-P_C(p^k)}+\norm{P_C(p^k)-P_C(Ax^k+b)} \\
      &=   \hJ(x^k,p^k)+2\norm{z^k},
  \end{align*}
  the second inequality in \eqref{eq.upper} tells us that for all $k \in \bN$ we hvae
  \[
    J(x^{k+1},0)\le \hJ(x^*,p^*)+2\norm{z^k} - (u^{k+1} )^T z^{k+1}  + \frac{1}{\mu} (q^k)^T (p^* - p^{k+1} )).
  \]
  By Lemma \ref{lem:adal 1} and Theorem \ref{thm:adal cvg}, the right-hand side of this inequality is bounded for all $k\in\bN$, and so, by Theorem \ref{thm:boundedness}, the sequence $\{x^k\}$ is bounded.  Corollary \ref{cor:adal cluster} then tells us that all cluster points of this sequence are solutions to \eqref{s-dist}.

  Now assume that (b) holds.  If the sequence $\{x^k\}$ is unbounded, then there is a subsequence $S\subset\bN$ and a vector $\bx\in X^\infty$ such that $\tnorm{x^k}\overset{S}{\to}\infty$ and $x^k/\tnorm{x^k}\overset{S}{\to} \bx$ with $\tnorm{\bx}=1$.  By Lemma \ref{lem:adal 1}, $\{p^k\}$ is bounded and, by Theorem \ref{thm:adal cvg}, $z^k\to 0$. Hence, $(Ax^k+b-p^k)/\tnorm{x^k}=z^k/\tnorm{x^k}\overset{S}{\to} 0$ so that $A\bx=0$. In addition, the sequence $\{\hJ(x^k,p^k)\}$ is bounded, which implies $\hJ(x^k,p^k)/\tnorm{x^k}^2\overset{S}{\to} 0$ so that $H\bx=0$. Moreover, since $H$ is positive semi-definite, $g^T(x^k/\tnorm{x^k})\le \hJ(x^k,p^k)/\tnorm{x^k}\overset{S}{\to} 0$ so that $g^T\bx\le 0$.  But then (b) implies that $\bx=0$. This contradiction implies that the sequence $\{x^k\}$ must be bounded. The result now follows from Corollary \ref{cor:adal cluster}.
\end{proof}

Note that, since $\ker(A)\subset A^{-1}C^\infty$, the condition given in (a) implies \eqref{weak boundedness}, and that \eqref{weak boundedness} is strictly weaker whenever $\ker(A)$ is strictly contained in $A^{-1}C^\infty$.

We conclude this section by stating a result for the case when $H$ is positive definite.  As has been observed, in such cases, the function $J_0$ is strongly convex and so the problem \eqref{s-dist} has a unique global solution $x^*$.  Hence, a proof paralleling that provided for Theorem \ref{thm:H pd implies cvg} applies to give the following result.

\begin{theorem}\label{thm:H pd implies adal cvg}
  Suppose that $H$ is positive definite and the sequence $\{(x^k,p^k,u^k)\}$ is generated by ADAL with initial point $x^0 \in X$.  Then, the problem \eqref{s-dist} has a unique global solution $x^*$ and $x^k\rightarrow x^*$.
\end{theorem}

\subsection{Complexity of ADAL}

In this subsection, we analyze the complexity of ADAL.  As was done for IRWA in Theorem~\ref{thm: IRWA complexity}, we show that ADAL requires at most $O(1/\veps^2)$ iterations to obtain an $\veps$-optimal solution to the problem \eqref{s-dist}.  In contrast to this result, some authors \cite{GoldMa12, GoldMaSche13} establish an $O(1/\veps)$ complexity for $\veps$-optimality for ADAL-type algorithms applied to more general classes of problems, which includes \eqref{s-dist}.  However, the ADAL decomposition employed by these papers involves subproblems that are as difficult as our problem \eqref{s-dist}, thereby rendering these decomposition unusable for our purposes.  On the other hand, under mild assumptions, the recent results in \cite{OADM} show that for a general class of problems, which includes \eqref{Jhat}, the ADAL algorithm employed here has $\hJ(x^k,p^k)$ converging to an $\veps$-optimal solution to \eqref{Jhat} with $O(1/\veps)$ complexity in an {\it ergodic sense} and $\|Ax+b-p\|_2^2$ converging to a value less than $\veps$ with $O(1/\veps)$ complexity.  This corresponds to an $O(1/\veps^2)$ complexity for $\veps$-optimality for problem \eqref{s-dist}.  As of this writing, we know of no result that applies to our ADAL algorithm that establishes a better iteration complexity bound for obtaining an $\veps$-optimal solution to \eqref{s-dist}.

We use results in \cite{OADM} to establish the following result.

\begin{theorem} \label{thm.complexity.adal}
  Consider the problem \eqref{s-dist} with $X=\Rn$ and suppose that the sequence $\{(x^k,p^k,u^k)\}$ is generated by ADAL with initial point $x^0 \in X$.  Then, under Assumption \ref{adal assumption}, in at most $O(1/\veps^2)$ iterations we have an iterate $x^{\bar k}$ with $k\le {\bar k}\le 2k-1$ that is $\veps$-optimal to \eqref{s-dist}, i.e., such that \eqref{eps opt sol} holds with $\tx=x^{\bar k}$.
\end{theorem}

The key results from \cite{OADM} used to prove this theorem follow. 

\begin{lemma}\label{fdes}\cite[Lemma 2]{OADM}
  Suppose that the sequence $\{(x^k, p^k, u^k)\}$ is generated by ADAL with initial point $x^0 \in X$, and, under Assumption~\ref{adal assumption}, let $(x^*, p^*, u^*)$ be the optimal solution of (\ref{Jhat}).  Then, for all $k \in \bN$, we have 
  \begin{align*}
    \hJ (x^{k+1}, p^{k+1}) - \hJ (x^*, p^*) 
      \le&\ \tfrac{\mu}{2}(\tnorm{u^k}^2 - \tnorm{u^{k+1}}^2) - \tfrac{1}{2\mu}\tnorm{Ax^k + b - p^{k+1}}^2 \\
         &\ + \tfrac{1}{2\mu}(\tnorm{Ax^* - Ax^k}^2 - \tnorm{Ax^* - Ax^{k+1}}^2).
  \end{align*}
\end{lemma}

\begin{lemma}\label{consdes}\cite[Theorem 2]{OADM}
  Suppose that the sequence $\{(x^k, p^k, u^k)\}$ is generated by ADAL with initial point $x^0 \in X$, and, under Assumption~\ref{adal assumption}, let $(x^*, p^*, u^*)$ be the optimal solution of (\ref{Jhat}). Then, for all $k \in \bN$, we have
  \[
    \tnorm{Ax^k + b - p^k}^2 + \tnorm{Ax^k - Ax^{k-1}}^2 \leq \tfrac{1}{k} [\tnorm{A(x^0-x^*)}^2 + \mu^2\tnorm{u^0-u^*}^2],
  \]
  i.e., in particular, we have
  \[
    \tnorm{Ax^k + b - p^k}^2 \leq \tfrac{1}{k} [\tnorm{A(x^0-x^*)}^2 + \mu^2\tnorm{u^0-u^*}^2].
  \]
\end{lemma}

\begin{rem}
  To see how the previous two lemmas follow from the stated results in \cite{OADM}, the table below provides a guide for translating between our notation and that of \cite{OADM}, which considers the problem
  \begin{equation}\label{OADM}
    \min_{x,z}\ f(x)+g(z)\ \mbox{subject to}\ Ax+Bz=c. 
  \end{equation}
  \[
    \begin{array}{|c|c|}
    \hline
    \mbox{Problem \eqref{Jhat}}&\mbox{Problem \eqref{OADM}}\\
    \hline
    (x,p) & (z,x)\\ \hline
    \varphi&g\\
    \hline
    \dist{\cdot }{C}&f\\
    \hline
    A & B \\
    \hline
    -I & A \\
    \hline
    -b&c\\
    \hline
    \end{array}
  \]
  For the results corresponding to our Lemmas \ref{fdes} and \ref{consdes}, \cite{OADM} requires $f$ and $g$ in \eqref{OADM} to be closed, proper, and convex functions.  In our case, the corresponding functions $\dist{\cdot}{C}$ and $\varphi$ satisfy these assumptions.
\end{rem}

By Lemma~\ref{lem.ubounded}, the sequence $\{\om^k\}$ is monotonically decreasing, meaning that $\{\|Ax^k-Ax^*\|_2^2\}$ and $\{\|u^k\|_2^2\}$ are bounded by some $\tau_1 > 0$ and $\tau_2>0$, respectively.  The proof of Theorem \ref{thm.complexity.adal} now follows as a consequence of the following lemma.

\begin{lemma}\label{ADALcomp}
  Suppose that the sequence $\{(x^k, p^k, u^k)\}$ is generated by ADAL with initial point $x^0 \in X$, and, under Assumption~\ref{adal assumption}, let $(x^*, p^*, u^*)$ be the optimal solution of (\ref{Jhat}).  Moreover, let $\bar k \in K := \{k, k+1, \dots, 2k - 1\}$ be such that $\hJ(x^{\bar k}, p^{\bar k}) = \min_{k\in K} \hJ(x^k, p^k)$.  Then,
 \[
   J_0(x^{\bar k}) - J_0(x^*) \leq \sqrt{ \frac{l(\tnorm{A(x^0-x^*)}^2 + \mu^2\tnorm{u^0-u^*}^2)}{k}} + \frac{\mu \tau_2 + \tau_1/\mu}{k} .
 \]
\end{lemma}
\begin{proof}
  Summing the inequality in Lemma \ref{fdes} for $j = k-1,\dots,2(k - 1)$ yields
  \begin{align}
        &\ \left(\sum_{j = k-1}^{2k- 2} \hJ(x^{j+1}, p^{j+1})\right) - k \hJ(x^*, p^*) \nonumber \\
    \leq&\ \tfrac{\mu}{2}(\tnorm{u^{k-1}}^2 - \tnorm{u^{2k- 1}}^2) + \tfrac{1}{2\mu}(\tnorm{Ax^* - Ax^{k-1}}^2 - \tnorm{Ax^* - Ax^{2k-1}}^2) \nonumber \\
    \leq&\ \mu \tau_2 + \tau_1/\mu. \label{com1}
  \end{align}
  Therefore,
  \begin{align}
    \hJ(x^{\bar k},p^{\bar k} ) - \hJ(x^*, p^*)
      & = \min_{k \leq j \leq 2k - 1}\hJ(x^j, p^j) - \hJ(x^*, p^*) \nonumber \\
      & \leq \tfrac{1}{k} \sum_{j = k-1}^{2k- 2} \hJ(x^{j+1}, p^{j+1}) - \hJ(x^*, p^*) \nonumber \\
      & \leq \tfrac{1}{k} (\mu \tau_2 + \tau_1/\mu), \label{com2}
  \end{align}
  where the last inequality follows from \eqref{com1}.

  Next, observe that for any $x \in \Rn$ and $p$, we have
  \begin{align}
    J_0(x) - \hJ(x, p) 
         =&\ \varphi(x) + \dist{Ax + b}{C} -  (\varphi(x) + \dist{p}{C}) \nonumber \\
         =&\ \dist{Ax + b}{C} - \dist{p}{C} \nonumber \\
      \leq&\ \norm{Ax + b - p} \nonumber \\
         =&\ \sum_{i\in\cI} \tnorm{A_i x + b_i - p_i} \nonumber \\
      \leq&\ \sqrt{l} \tnorm{Ax + b - p}, \label{eq.fact2}
  \end{align}
  where the first inequality follows since $| \dist{z}{C} - \dist{w}{C}|\le\Vert z-w\Vert$, and the second follows by Jensen's inequality.  Combining \eqref{com2} and \eqref{eq.fact2} gives
  \[
    \begin{aligned}
      J_0(x^{\bar k}) - J_0(x^*)
           =&\ J_0(x^{\bar k}) - \hJ(x^*, p^*)\\
           =&\ J_0(x^{\bar k}) - \hJ(x^{\bar k}, p^{\bar k}) + \hJ(x^{\bar k}, p^{\bar k}) - \hJ(x^*, p^*) \\
        \leq&\ \sqrt{l} \tnorm{Ax^{\bar k} + b - p^{\bar k}} + \frac{\mu \tau_2 + \tau_1/\mu}{k} \\
        \leq&\ \sqrt{\frac{l(\tnorm{A(x^0-x^*)}^2 + \mu^2\tnorm{u^0-u^*}^2)}{ k} } + \frac{\mu \tau_2 + \tau_1/\mu}{k},
   \end{aligned}
  \]
  where the second inequality follows by Lemma~\ref{consdes} and the fact that $\bar k \ge k$.
\end{proof}

\section{Nesterov Acceleration}

In order to improve the performance of both IRWA and ADAL, one can use an acceleration technique due to Nesterov \cite{Nest83}.  For the ADAL algorithm, we have implemented the acceleration as described in \cite{GoldDonoSetz10}, and for the IRWA algorithm the details are given below.  We conjecture that each accelerated algorithm requires $O(1/\veps)$ iterations to produce an $\veps$-optimal solution to \eqref{s-dist}, but this remains an open issue.
\medskip

\underline{\bf{IRWA with Nesterov Acceleration}}
\begin{enumerate}
  \item[Step 0:] (Initialization) Choose an initial point $x^0\in X$, an initial relaxation vector $\eps^0 \in\bR^l_{++}$, and scaling parameters $\eta\in(0,1)$, $\gamma >0$, and $M>0$.  Let $\sigma \geq 0$ and $\sigma' \geq 0$ be two scalars which serve as termination tolerances for the stepsize and relaxation parameter, respectively.  Set $k:=0$, $y^0 := x^0$, and $t_1 := 1$.
  \item[Step 1:] (Solve the re-weighted subproblem for $x^{k+1}$)\\
  \noindent
  Compute a solution $x^{k+1}$ to the problem
  \begin{equation*}
    \cG{(y^k,\eps^k)}:\quad \min_{x\in X} \hat G_{(y^k,\eps^k)}(x).
  \end{equation*}
  Let 
  \begin{align*}
     t_{k+1} & := \tfrac{1 + \sqrt{1 + 4(t^k)^2}}{2} \\
     \mbox{and}\quad 
     y^{k+1} & :=  x^{k+1} + \tfrac{t^k - 1}{t_{k+1}}( x^{k+1} -  x^k).
  \end{align*}
  \item[Step 2:] (Set the new relaxation vector $\eps^{k+1}$)\\
  \noindent
  Set 
  \[
    \tilde q_i^k:=A_i (x^{k+1} - y^k)
    \quad\mbox{and}\quad
    \tilde r_i^k:=(I-\Pri)(A_i y^k + b_i)\quad \forall\, i\in\cI_0. 
  \]
  If  
  \begin{equation*}
    \tnorm{\tilde q_i^k} \leq  M\Big[ \tnorm{\tilde r_i^k}^2 + (\eps_i^k)^2\Big]^{\frac{1}{2}+\gamma} \quad \forall\, i \in \cI, 
  \end{equation*}
  then choose $\eps^{k+1} \in(0,\ \eta \eps^k]$; else, set $\eps^{k+1} := \eps^k$.  If $J(y^{k+1}, \eps^{k+1}) > J(x^{k+1}, \eps^{k+1})$, then set $y^{k+1} := x^{k+1}$.
  \item[Step 3:] (Check stopping criteria)\\
  \noindent
  If $\norm{x^{k+1} - x^k}_2 \leq \sigma$ and $\norm{\eps^{k}}_2 \leq \sigma'$, then stop; else, set $k := k+1$ and go to Step 1. 
\end{enumerate}
\medskip

In this algorithm, the intermediate variable sequence $\{y^k\}$ is included.  If $y^{k+1}$ yields an objective function value worse than $x^{k+1}$, then we re-set $y^{k+1}:=x^{k+1}$.  This modification preserves the global convergence properties of the original version since
\begin{align}
      &\ J(x^{k+1}, \eps^{k+1}) - J(x^{k}, \eps^{k}) \nonumber \\
     =&\ J(x^{k+1}, \eps^{k+1}) - J(y^{k}, \eps^{k}) + J(y^{k}, \eps^{k})  - J(x^{k}, \eps^{k}) \nonumber \\
  \leq&\ J(x^{k+1}, \eps^{k}) - J(y^{k}, \eps^{k}) \nonumber \\
  \leq&\ -\tfrac{1}{2}(x^{k+1} - y^{k})^T\tilde A^T W_k\tilde A(x^{k+1} - y^{k}) \label{e2} \\
     =&\ -\tfrac{1}{2}(\tilde q^k)^T W_k \tilde q^k,\nonumber
 \end{align}
 where the inequality \eqref{e2} follows from Lemma \ref{Gfun}.  Hence, $\frac{1}{2}(\tilde q^k)^T W_k \tilde q^k$ is summable, as was required for Lemma \ref{irwa properties} and Theorem \ref{main}. 

\section{Application to Systems of Equations and Inequalities}

In this section, we discuss how to apply the general results from \S\ref{sec.irwa} and \S\ref{sec.adal} to the particular case when $H$ is positive definite and the system $Ax+b\in C$ corresponds a system of equations and inequalities.  Specifically, we take $l=m$, $X=\Rn$, $C_i=\{0\}$ for $i\in\{1,\dots,s\}$, and $C_i=\R_-$ for $i\in\{s+1,\dots,m\}$ so that $C:=\{0\}^s\times\R_-^{m-s}$ and
\begin{align}
  J_0(x)
    &= \varphi(x)+\odist{Ax+b}{C} \nonumber \\
    &=\varphi(x)+\sum_{i=1}^{s}|A_ix+b_i| + \sum_{i=s+1}^m(A_ix+b_i)_+. \label{experimental J}
\end{align}
The numerical performance of both IRWA and ADAL on problems of this type will be compared in the following section.  For each algorithm, we examine performance relative to a stopping criteria, based on percent reduction in the initial duality gap.  It is straightforward to show that, since $H$ is positive definite, the Fenchel-Rockafellar dual \cite[Theorem 31.2]{RTR} to \eqref{s-dist} is
\begin{equation}\label{dual problem}
  \begin{array}{ll}
    \displaystyle{\mathop{\mathrm{minimize}}_u}&\ \half(g+A^Tu)^TH^{-1}(g+A^Tu)-b^Tu+\sum_{i\in\cI}\support{u_i}{C_i} \\
    \mbox{subject to}&\ u_i\in\bB_2\ \forall\, i\in\cI,
  \end{array}
\end{equation}
which in the case of \eqref{experimental J} reduces to
\[
  \begin{array}{ll}
    \displaystyle{\mathop{\mathrm{minimize}}_u}&\ \half(g+A^Tu)^TH^{-1}(g+A^Tu)-b^Tu\\
    \mbox{subject to}&\ -1\le u_i\le 1,\ i=1,\dots,s\\
                     &\, \quad 0\le u_i\le1,\ i=s+1,\dots,m.
  \end{array}
\]

In the case of linear systems of equations and inequalities, IRWA can be modified to improve the numerical stability of the algorithm.  Observe that if both of the sequences $|r_i^k|$ and $\eps_i^k$ are driven to zero, then the corresponding weight $w_i^k$ diverges to $+\infty$, which may slow convergence by unnecessarily introducing numerical instability.  Hence, we propose a modification that addresses those iterations and indices $i\in\{s+1,\dots,m\}$ for which $(A_ix^k+b_i)_-<0$, i.e., those inequality constraint indices corresponding inequality constraints that are strictly satisfied (inactive).  For such indices, it is not necessary to set $\eps_i^{k+1}<\eps_i^k$.  There are many possible approaches to address this issue, one of which is given in the algorithm given below.
\medskip

\underline{\bf{IRWA for Systems of Equations and Inequalities}}
\begin{enumerate}
  \item[Step 0:] (Initialization)
  Choose an initial point $x^0\in X$, initial relaxation vectors $\heps^0=\eps^0 \in\bR^l_{++}$, and scaling parameters $\eta\in(0,1)$, $\gamma >0$, and $M>0$.  Let $\sigma \geq 0$ and $\sigma' \geq 0$ be two scalars which serve as termination tolerances for the stepsize and relaxation parameter, respectively.  Set $k:=0$.
  \item[Step 1:] (Solve the re-weighted subproblem for $x^{k+1}$)\\
  \noindent
  Compute a solution $x^{k+1}$ to the problem
  \[
    \cG{(x^k,\eps^k)}:\quad \min_{x\in X} \hat G_{(x^k,\eps^k)}(x).
  \]
  \item[Step 2:] (Set the new relaxation vector $\eps^{k+1}$)\\
  \noindent
  Set 
  \[
    q_i^k:=A_i (x^{k+1} - x^k)
    \quad\mbox{and}\quad
    r_i^k:=(I-\Pri)(A_i x^k + b_i)\quad \forall\, i=0,\dots,m. 
  \]
  If
  \begin{equation}\label{step 2b inequality}
    \tnorm{q_i^k} \leq  M\Big[ \tnorm{r_i^k}^2 + (\eps_i^k)^2\Big]^{\frac{1}{2}+\gamma}\quad \forall\, i = 1, \dots, m, 
  \end{equation}
  then choose $\heps^{k+1} \in(0,\ \eta \heps^k]$ and, for $i=1,\dots, m$, set 
  \[
    \eps_i^{k+1} := \begin{cases} \heps_i^{k+1}&\mbox{, $i=1,\dots,s$,} \\ \eps_i^k&\mbox{, $i>s$ and $(A_ix^k+b_i)_-\le -\heps_i^k$,} \\ \heps_i^{k+1}&\mbox{, otherwise.} \end{cases}
  \]
  Otherwise, if \eqref{step 2b inequality} is not satisfied, then set $\heps^{k+1}:=\heps^{k}$ and $\eps^{k+1}:=\eps^{k}$.
  \item[Step 3:] (Check stopping criteria)\\
  \noindent
  If $\norm{x^{k+1} - x^k}_2 \leq \sigma$ and $\norm{\heps^{k}}_2 \leq \sigma'$, then stop; else, set $k := k+1$ and go to Step 1.
\end{enumerate}
\medskip

\begin{rem}
  In Step 2 of the algorithm above, the updating scheme for $\eps$ can be modified in a variety of ways. For example, one can also take $\eps_i^{k+1}:=\eps_i^{k}$ when $i>s$ and $(A_ix^k+b_i)_-<0$.
\end{rem}

This algorithm yields the following version of Lemma \ref{irwa properties}.

\begin{lemma}\label{irwa properties 2}
  Suppose that the sequence $\{(x^k,\eps^k)\}$ is generated by IRWA for Systems of Equations and Inequalities with initial point $x^0 \in X$ and relaxation vector $\epsilon^0 \in \bR^l_{++}$, and, for $k\in\bN$, let $q_i^k$ and $r_i^k$ for $i\in\cI_0$ be as defined in Step 2 of the algorithm with 
  \[
    q^k:=((q_0^k)^T,\dots,(q_l^k)^T)^T
    \quad\mbox{and}\quad
    r^k:=((r_0^k)^T,\dots,(r_l^k)^T)^T.
  \]
  Moreover, for $k \in \bN$, define 
  \[
    w^k_i:=w_i(x^k,\eps^k)\ \mbox{for}\ i\in\cI_0
    \quad\mbox{and}\quad
    W_k:=W(x^k,\eps^k),
  \]
  and set $S:=\set{k}{\eps^{k+1}\le\eta\eps^k}$.  Then, the sequence $\{J(x^{k}, \eps^k)\}$ is monotonically decreasing.  Moreover, either $\inf_{k\in\bN} J(x^k,\eps^k) = -\infty$, in which case $\inf_{x\in X} J_0(x) = -\infty$, or the following hold:
  \begin{enumerate}
    \item $\sum_{k=0}^\infty (q^k)^TW_kq^k< \infty$.
    \item $\heps^k\rightarrow 0$ and $H(x^{k+1}-x^k)\to 0$.
    \item $W_kq^k\overset{S}{\rightarrow}0$.
    \item $w_i^kr_i^k=r^k_i/\sqrt{\tnorm{r^k_i}^2+\eps^k_i}\in\bB_2\cap\ncone{\Pri(A_ix^k+b_i)}{C_i},\, i\in\cI,\, k\in\bN$.
    \item $-\tA^TW_kq^k\in (\nabla\varphi(x^k)+\sum_{i\in\cI}A_i^Tw_i^kr_i^k) + \ncone{x^{k+1}}{X},\, k\in\bN$.
    \item If $\{\dist{Ax^k+b}{C}\}_{k\in S}$ is bounded, then $q^k\overset{S}{\rightarrow}0$. 
  \end{enumerate}
\end{lemma}
\begin{proof}
  Note that Lemma~\ref{Gfun} still applies since it is only concerned with properties of the functions $\hG$ and $J$.  In addition, note that
  \[
    \heps^{k+1}\le\heps^{k}\quad\mbox{and}\quad\heps^{k+1}\le\eps^{k+1}\le\eps^{k}\quad \forall\, k\geq1.
  \]
  With these observations, the proof of this lemma follows in precisely the same way as that of Lemma \ref{irwa properties}, except that in Part (2) $\{\heps^k\}$ replaces $\{\eps^k\}$.
\end{proof}

With Lemma \ref{irwa properties 2}, it is straightforward to show that the convergence properties described in Theorems \ref{main} and \ref{thm:H pd implies cvg} also hold for the version of IRWA in this section.

\section{Numerical Comparison of IRWA and ADAL}

In this section, we compare the performance of our IRWA and ADAL algorithms in a set of three numerical experiments.  The first two experiments involves cases where $H$ is positive definite and the desired inclusion $Ax+b\in C$ corresponds to a system of equations and inequalities.  Hence, for these experiments, we employ the version of IRWA as described for such systems in the previous section.  In the first experiment, we fix the problem dimensions and compare the behavior of the two algorithms over $500$ randomly generated problems.  In the second experiment, we investigate how the methods behave when we scale up the problem size.  For this purpose, we compare performance over 20 randomly generate problems of increasing dimension.  The algorithms were implemented in Python using the NumPy and SciPy packages; in particular, we used the versions Python 2.7, Numpy 1.6.1, SciPy 0.12.0 \cite{scipy,scipy2}.  In both experiments, we examine performance relative to a stopping criteria based on percent reduction in the initial duality gap.  In IRWA, the variables $\tu^k:=W_kr^k$ are always dual feasible, i.e.,
\[
  \tu_i\in\bB_2\cap\dom{\support{\cdot}{C_i}}\quad \forall\, i\in\cI
\]
(recall Lemma \ref{irwa properties}(4)), and these variables constitute our $k$th estimate to the dual solution.  On the other hand, in ADAL, the variables $\hu^k=u^k-\frac{1}{\mu}q^k$ are always dual feasible (recall Lemma~\ref{lem:adal 1}), so these constitute our $k$th estimate to the dual solution for this algorithm.  The duality gap at any iteration is the sum of the primal and dual objectives at the current primal-dual iterates.

In both IRWA and ADAL, we solve the subproblems using CG, which is terminated when the $\ell_2$-norm of the residual is less than $10\%$ of the norm of the initial residual.  At each iteration, the CG algorithm is initiated at the previous step $x^{k-1}$.  In both experiments, we set $x^0:=0$, and in ADAL we set $u^0:=0$.  It is worthwhile to note that we have observed that the performance of IRWA is sensitive to the initial choice of $\eps^0$ while ADAL is sensitive to $\mu$.  We do not investigate this sensitivity in detail when presenting the results of our experiments, and we have no theoretical justification for our choices of these parameters.  However, we empirically observe that these values should increase with dimension.  For each method, we have chosen an automatic procedure for initializing these values that yields good overall performance.  The details are given in the experimental descriptions.  More principled methods for initializing and updating these parameters is the subject of future research. 

In the third experiment, we apply both algorithms to an $l_1$ support vector machine (SVM) problem.  Details are given in the experimental description. In this case, we use the stopping criteria as stated along with the algorithm descriptions in the paper, i.e., not criteria based on a percent reduction in duality gap.   In this experiment, the subproblems are solved as in the first two experiments with the same termination and warm-start rules.
\medskip

\noindent
{\bf First Experiment:} 
\begin{figure*}
  \begin{tabular}{cc}
    \includegraphics[scale=.5]{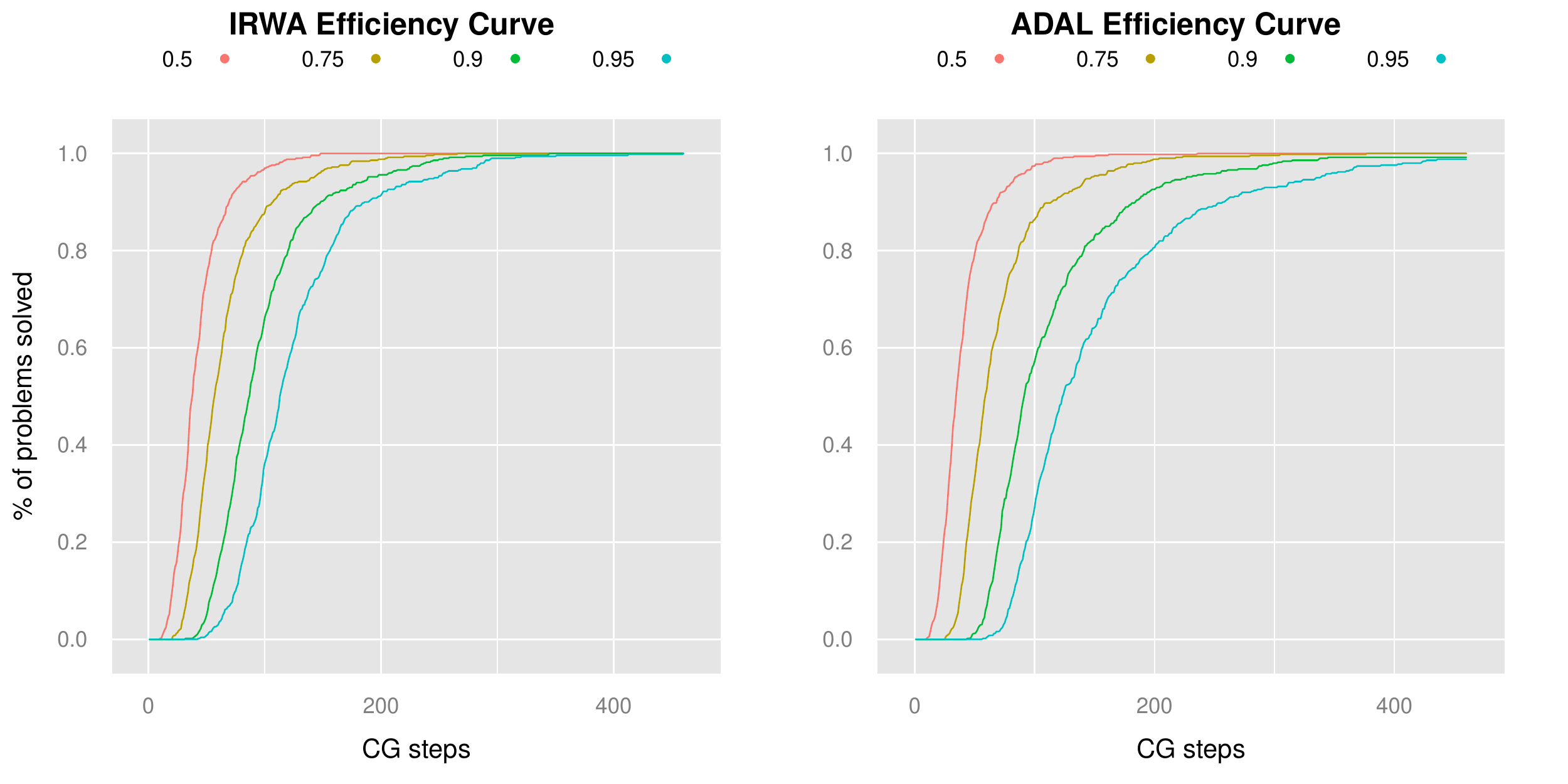}
  \end{tabular}
  \caption{Efficiency curves for IRWA (left panel) and ADAL (right panel).  The percentage of the 500 problems solved is plotted verses the total number of CG steps.  IRWA terminated in fewer than 460 CG steps on all problems.  ADAL required over 460 CG steps on 8 of the problems.} 
  \label{efficiency curves}
\end{figure*}
\begin{figure}[!ht]
  \begin{center}
    \includegraphics[scale = 0.4]{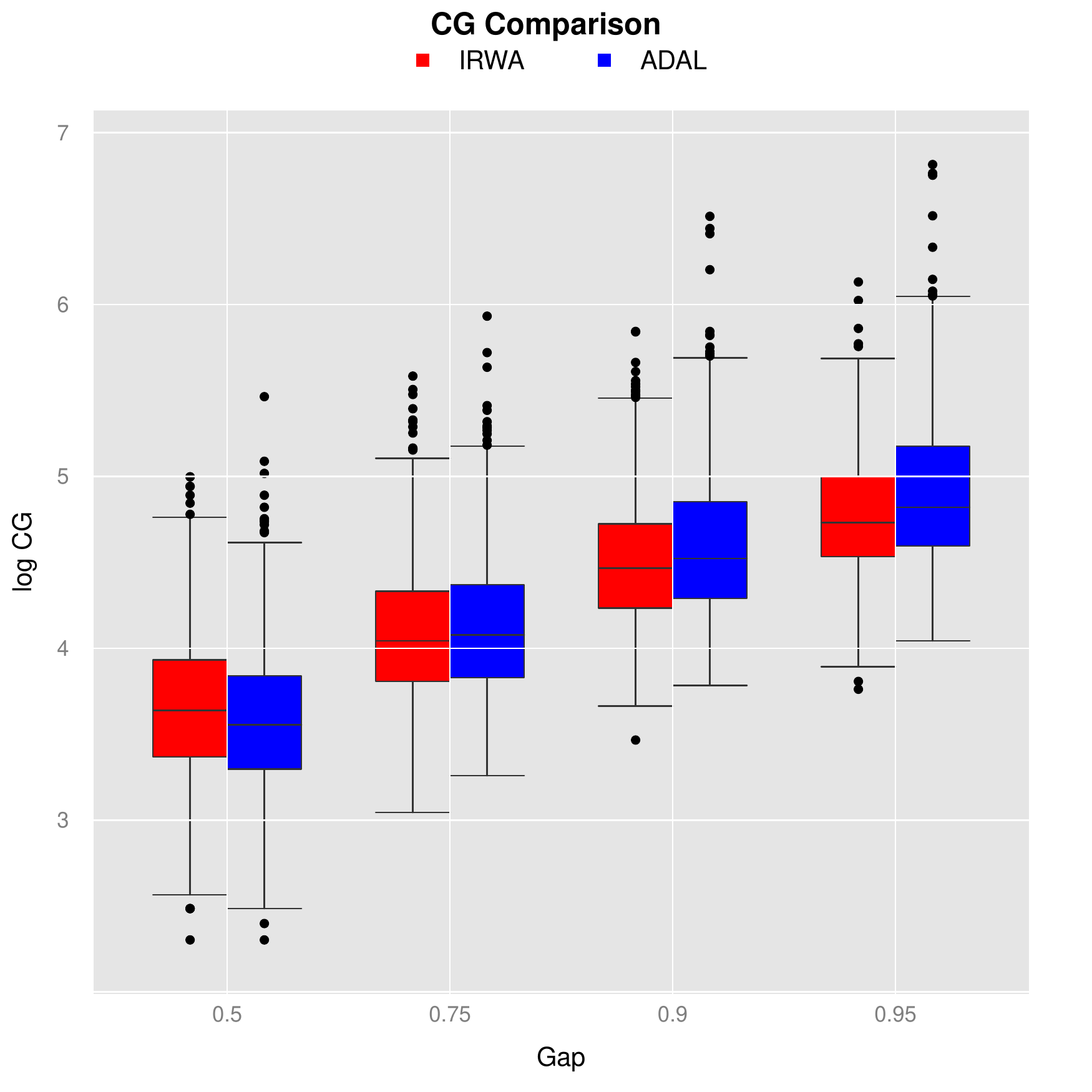}
  \end{center}
  \caption{Box plot of CG steps for IRWA (red) and ADAL (blue) for each duality gap threshold.}
  \label{fig:cgRegular}
\end{figure}
In this experiment, we randomly generated $500$ instances of problem \eqref{experimental J}.  For each, we generated $A \in \bR^{600 \times 1000}$ and chose $C$ so that the inclusion $Ax+b\in C$ corresponded to 300 equations and 300 inequalities.  Each matrix $A$ is obtained by first randomly choosing a mean and variance from the integers on the interval $[1,10]$ with equal probability.  Then the elements of $A$ are chosen from a normal distribution having this mean and variance.  Similarly, each of the vectors $b$ and $g$ are constructed by first randomly choosing integers on the intervals $[-100,100]$ for the mean and $[1,100]$ for the variance with equal probability and then obtaining the elements of these vectors from a normal distribution having this mean and variance.  Each matrix $H$ had the form $H=0.1 I +LL^T$ where the elements of $L\in\R^{n\times n}$ are chosen from a normal distribution having mean $1$ and variance $2$.  For the input parameters for the algorithms, we chose $\eta:=0.6$, $M := 10^4$, $\gamma := \frac{1}{6}$, $\mu:=100$, and $\eps^0_i := 2000$ for each $i \in \cI$.  Efficiency curves for both algorithms are given in Figure~\ref{efficiency curves}, which illustrates the percentage of problems solved verses the total number of CG steps required to reduce the duality gap by 50, 75, 90 and 95 percent.  The greatest number of CG steps required by IRWA was $460$ when reducing the duality gap by $95\%$.  ADAL stumbled at the $95\%$ level on 8 problems, requiring $609, 494, 628, 674, 866, 467, 563, 856, 676$ and $911$ CG steps for these problems.  Figure~\ref{fig:cgRegular} contains a box plot for the log of the number of CG iterations required by each algorithm for each of the selected accuracy levels.  Overall, in this experiment, the methods seem comparable with a slight advantage to IRWA in both the mean and variance of the number of required CG steps.
\medskip

\noindent
{\bf Second Experiment:}
\begin{figure*}
  \begin{tabular}{cc}
    {\includegraphics[scale=0.36]{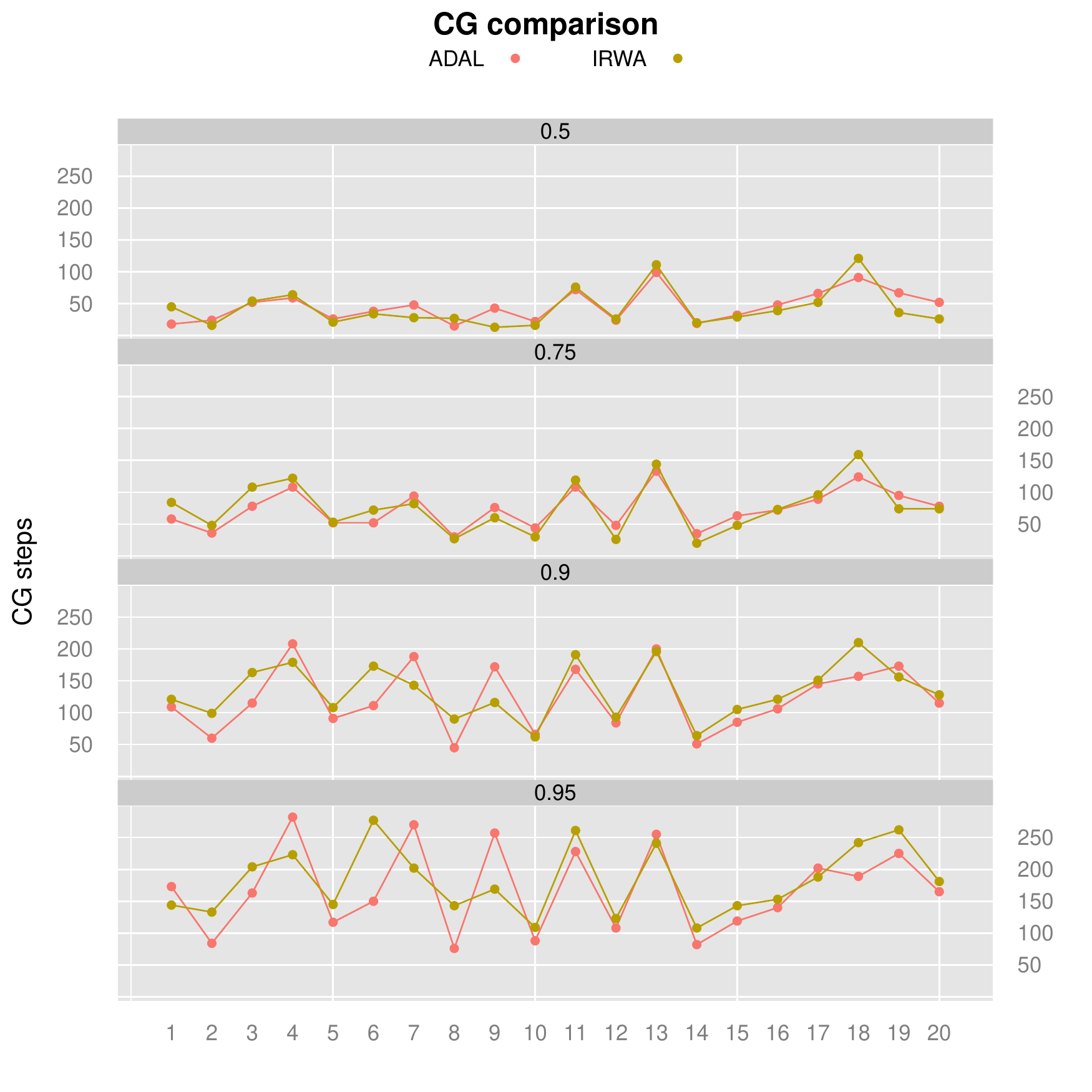}}
    {\includegraphics[scale=0.36]{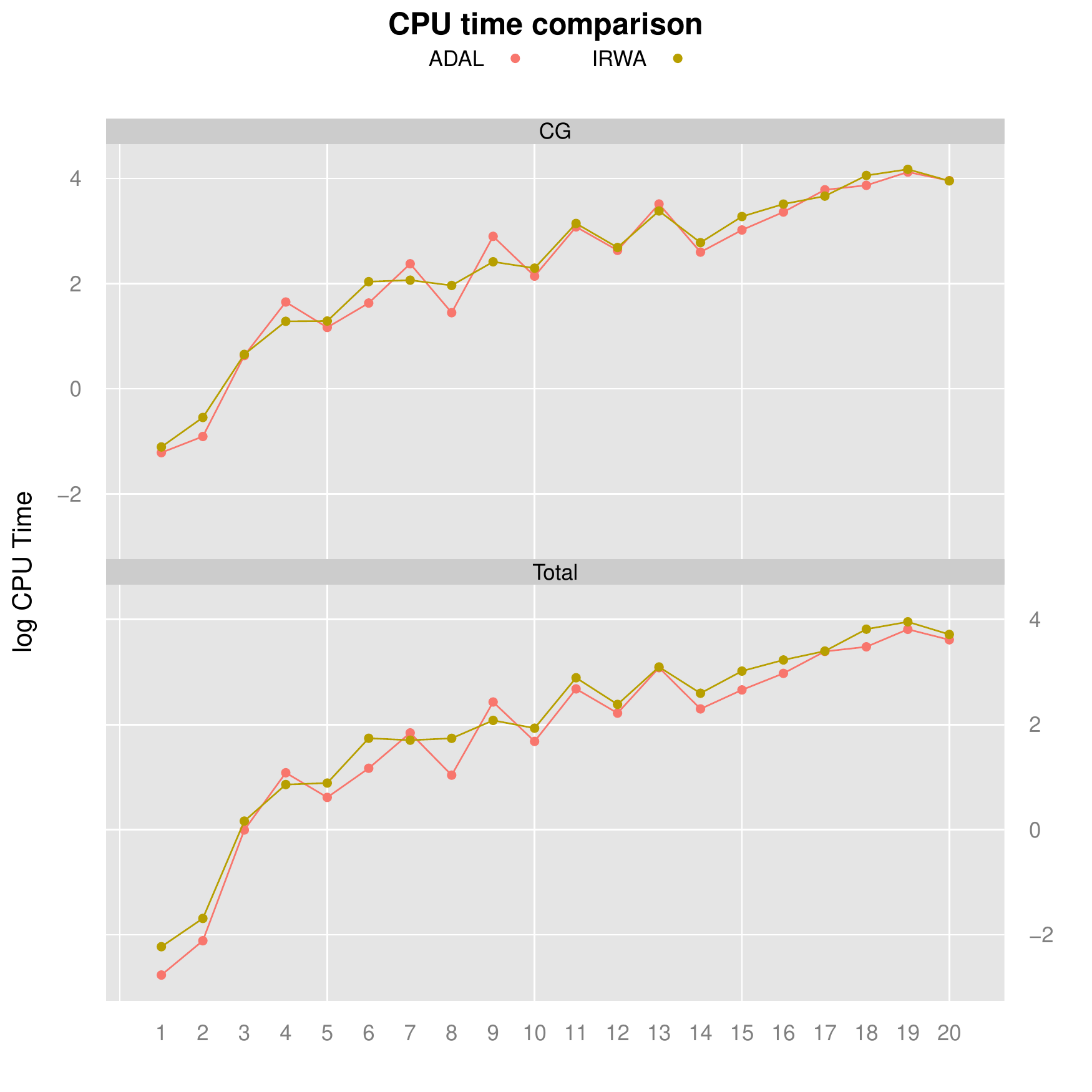}} 
  \end{tabular}
  \caption{Increase in CG steps for each duality gap (left panel) and CPU time (right panel) for IRWA and ADAL as dimension is increased.} 
  \label{cg and cpu}
\end{figure*}
In the second experiment, we randomly generated 20 problems of increasing dimention.  The numbers of variables were chosen to be $n=200+500(j-1),\ j=1,\dots,20$, where for each we set $m:=n/2$ so that the inclusion $Ax+b\in C$ corresponds to equal numbers of equations and inequalities.  The matrix $A$ was generated as in the first experiment.  Each of the vectors $b$ and $g$ were constructed by first choosing integers on the intervals $[-200,200]$ for the mean and $[1,200]$ for the variance with equal probability and then obtaining the elements of these vectors from a normal distribution having this mean and variance.  Each matrix $H$ had the form $H=40 I +LDL^T$, where $L\in\R^{n\times k}$ with $k=8$ and $D$ was diagonal.  The elements of $L$ were constructed in the same way as those of $A$, and those of $D$ were obtained by sampling from the inverse gamma distribution $f(x):=\frac{b^a}{\Gamma(a)} x^{-a - 1} e^{-b/x}$ with $a = 0.5,\ b = 1$.  We set $\eta:=0.5$, $M := 10^4$, and $\gamma := \frac{1}{6}$, and for each $j=1,\dots,20$ we set $\eps^0_i:=10^{2+1.3\ln(j+10)}$ for each $i=1,\dots,m$, and $\mu:=500(1+j)$.  In Figure~\ref{cg and cpu}, we present two plots showing the number of CG steps and the log of the CPU times versus variable dimensions for the two methods.  The plots illustrate that the algorithms performed similarly in this experiment.
\medskip

\noindent
{\bf Third Experiment:}
In this experiment, we solve the $l_1$-SVM problem as introduced in \cite{JZhu}.  In particular, we consider the exact penalty form
\begin{align}
  \min_{\beta \in \bR^n} \ \sum_{i = 1}^m \left(1 - y_i\left(\sum_{j = 1}^n x_{ij} \beta_j\right)\right)_+ + \lambda \norm{\beta}_1,
\end{align}
where $\{(\pmb{x_i}, y_i)\}_{i = 1}^m$ are the training data points with $\pmb{x}_i \in \bR^n$ and $y_i \in \{-1, 1\}$ for each $i = 1,\dots,m$, and $\lambda$ is the penalty parameter.  In this experiment, we randomly generated 40 problems in the following way.  First, we sampled an integer on $[1,5]$ and another on $[6,10]$, both from uniform distributions.  These integers were taken as the mean and standard deviation of a normal distribution, respectively.  We then generated an $m\times s$ component-wise normal random matrix $T$, where $s$ was chosen to be $19 + 2j,\ j = 0, 1, \dots, 39$ and $m$ was chosen to be $200 + 10j, \ j = 0, 1, \dots, 39$.  We then generated an $s$-dimensional integer vector $\hat \beta$ whose components were sampled from the uniform distribution on the integers between $-100$ and $100$.  Then, $y_i$ was chosen to be the sign of the $i$-th component of $T \hat\beta$.  In addition, we generated an $m\times t$ i.i.d.~standard normal random matrix $R$, where $t$ was chosen to be $200 + 30j, \ j = 0, 1, \dots, 39$.   Then, we let $[\pmb{x}_1, \pmb{x}_2, \dots, \pmb{x}_m]^T := [T, R]$.  For all 40 problems, we fixed the penalty parameter at $\lambda = 50$.  In this application, the problems need to be solved exactly, i.e., a percent reduction in duality gap is insufficient.  Hence, in this experiment, we use the stopping criteria as described in Step 3 of both IRWA and ADAL.  For IRWA, we set $\eps_i^0 := 10^4$ for all $i \in \cI$, $\eta := 0.7$, $M := 10^4$, $\gamma := \frac{1}{6}$, $\sigma := 10^{-4}$ and $\sigma' := 10^{-8}$.  For ADAL, we set $\mu := 1$, $\sigma := 0.05$ and $\sigma'' := 0.05$.  We also set the maximum iteration limit for ADAL to 150.  Both algorithms were initialized at $\beta := 0$.  Figure \ref{fig:svm1} has two plots showing the objective function values of both algorithms at termination, and the total CG steps taken by each algorithm.  These two plots show superior performance for IRWA when solving these 40 problems.

Based on how the problems were generated, we would expect the non-zero coefficients of the optimal solution $\beta$ to be among the first $s = 19 + 2j, \ j = 0, \dots, 39$ components corresponding to the matrix $T$.  To investigate this, we considered ``zero'' thresholds of $10^{-3}, 10^{-4}$ and $10^{-5}$; i.e., we considered a component as being ``equal'' to zero if its absolute value was less than a given threshold.  Figure \ref{fig:svm2} shows a summary of the number of unexpected zeros for each algorithm.  These plots show that IRWA has significantly fewer false positives for the nonzero components, and in this respect returned preferable sparse recovery results over ADAL in this experiment.

\begin{figure}[!ht]
 \begin{center}
  \includegraphics[scale = 0.4]{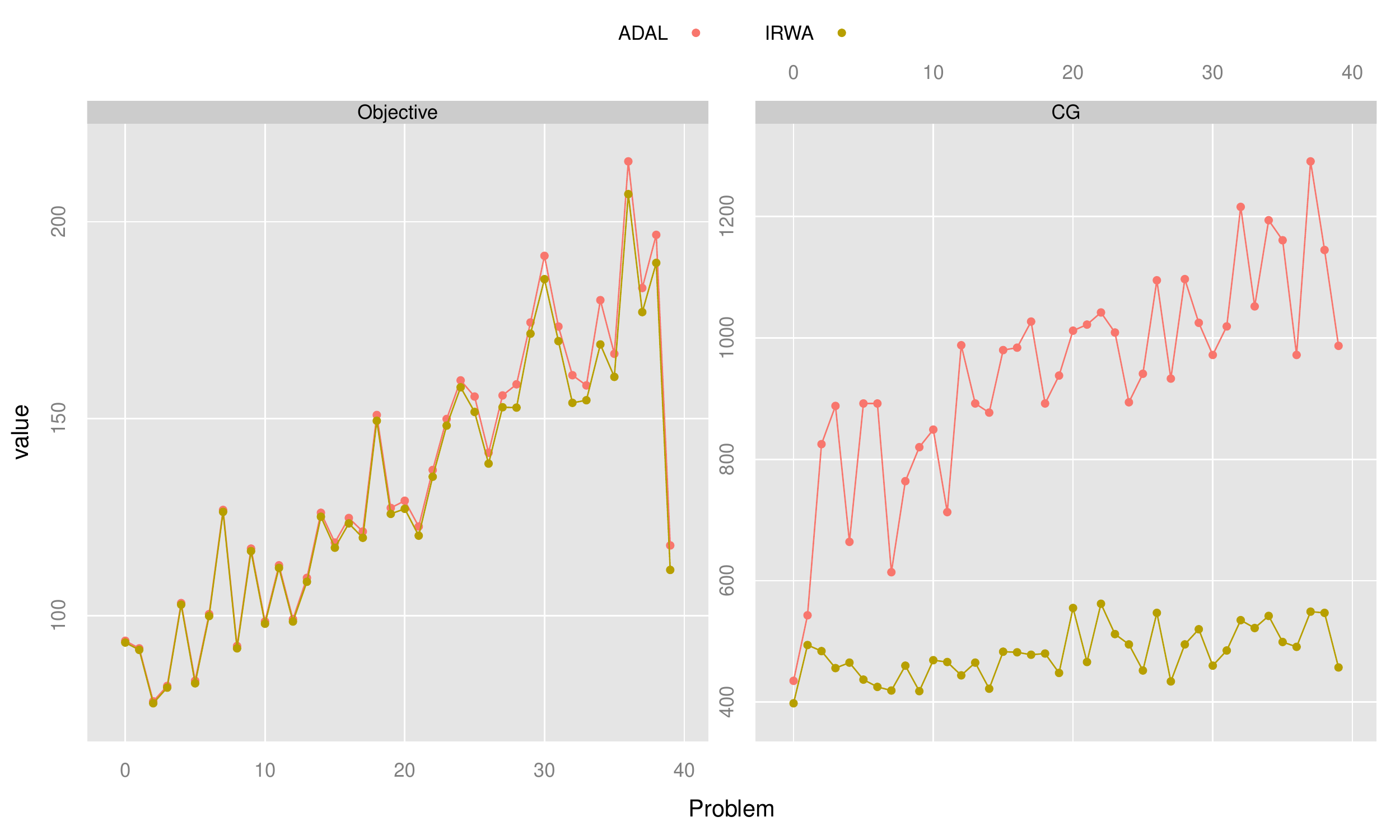}
 \end{center}
 \caption{In all 40 problems, IRWA obtains smaller objective function values with fewer CG steps.}
 \label{fig:svm1}
\end{figure}

\begin{figure}[!ht]
 \begin{center}
  \includegraphics[scale = 0.4]{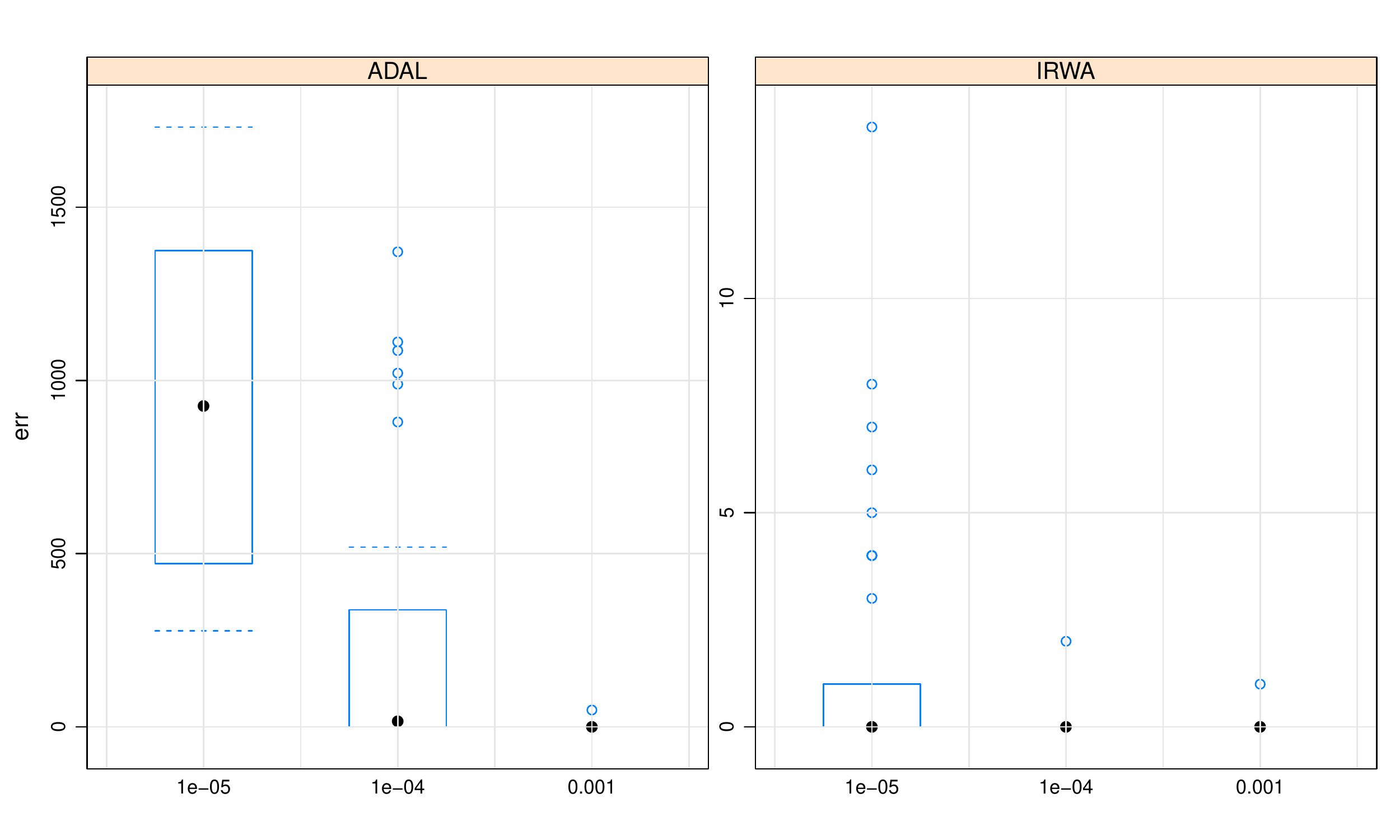}
 \end{center}
 \caption{For both thresholds $10^{-4}$ and $10^{-5}$, IRWA yields fewer false positives in terms of the numbers of ``zero'' values computed.  The numbers of false positives is similar for the threshold $10^{-3}$.  At the threshold $10^{-5}$, the difference in recovery is dramatic with IRWA always having fewer than 14 false positives while ADAL has a median of about $1000$ false positives.}
 \label{fig:svm2}
\end{figure}

Finally, we use this experiment to demonstrate Nesterov's acceleration for IRWA.  The effect on ADAL has already been shown in \cite{GoldDonoSetz10}, so we only focus on the effect of accelerating IRWA.  The 40 problems were solved using both IRWA and accelerated IRWA with the parameters stated above.  Figure~\ref{fig:effectNes} shows the differences between the objective function values ($\frac{\mbox{normal} - \mbox{accelerated}}{\mbox{accelerated}}\times 100$) and the number of CG steps (normal $-$ accelerated) needed to converge.  The graphs show that accelerated IRWA performs significantly better than unaccelerated IRWA in terms of both objective function values obtained and CG steps required.

\begin{figure}[!ht]
 \begin{center}
  \includegraphics[scale = 0.4]{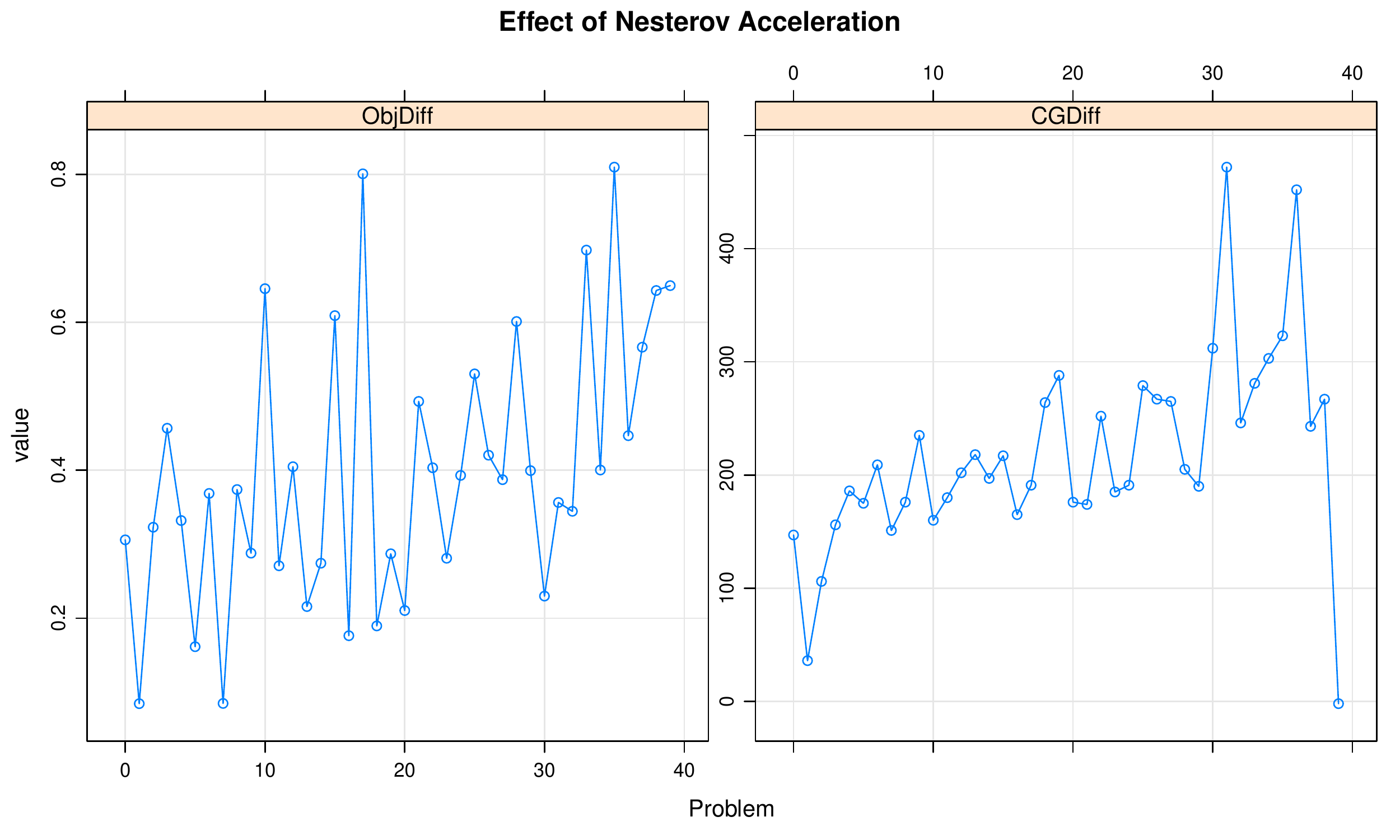}
 \end{center}
 \caption{Differences in objective function values (left panel) obtained by normal and accelerated IRWA ($\frac{\mbox{normal} - \mbox{accelerated}}{\mbox{accelerated}}\times 100$), and differences in numbers of CG steps (right panel) required to converge to the objective function value in the left panel ($\mbox{normal} - \mbox{accelerated}$).  Accelerated IRWA always converged to a point with a smaller objective function value, and accelerated IRWA typically required fewer CG steps.  (There was only one exception, the last problem, where accelerated IRWA required two more CG steps.)}
 \label{fig:effectNes}
\end{figure}

\section{Conclusion}

In this paper, we have proposed, analyzed, and tested two matrix-free solvers for approximately solving the exact penalty subproblem \eqref{s-dist}.  The primary novelty of our work is a newly proposed iterative re-weighting algorithm (IRWA) for solving such problems involving arbitrary convex sets of the form \eqref{prod C}.  In each iteration of our IRWA algorithm, a quadratic model of a relaxed problem is formed and solved to determine the next iterate.  Similarly, the alternating direction augmented Lagrangian (ADAL) algorithm that we present also has as its main computational component the minimization of a convex quadratic subproblem.  Both solvers can be applied in large scale settings, and both can be implemented matrix-free.

Variations of our algorithms were implemented and the performance of these implementations were tested.  Our test results indicate that both types of algorithms perform similarly on many test problems.  However, a test on an $\ell_1$-SVM problem illustrates that in some applications the IRWA algorithms can have superior performance.  While the accelerated version of both methods is the preferred implementation, we have provided global convergence and complexity results for unaccelerated variants of the algorithms.  Complexity results for accelerated versions remains an open issue.

\bibliographystyle{plain}
\bibliography{biblio}
\end{document}